\documentclass[12pt,reqno,english]{amsart}

\usepackage{amsmath, amssymb, amscd, amsthm, amsfonts}
\usepackage{graphicx,xcolor,color}
\usepackage[linktocpage,colorlinks,linkcolor=blue,anchorcolor=blue,citecolor=blue,urlcolor=blue,pagebackref]{hyperref}
\usepackage{mathtools,enumitem}
\usepackage[T1]{fontenc}
\usepackage[utf8]{inputenc}
\usepackage{float}
\usepackage{mfirstuc}
\usepackage{amsmath, amssymb, amscd, amsthm, amsfonts, bm}
\usepackage{amssymb,amsthm,amsfonts,amsbsy,latexsym,amsxtra}
\usepackage{mathtools,enumitem}

\usepackage{float}
\usepackage{soul}
\usepackage{mathrsfs} 
\usepackage[edges]{forest}

\newcommand{\adapt}{\mathsf{adapt}}
\numberwithin{equation}{section}
\theoremstyle{plain}
\usepackage[letterpaper]{geometry}
\renewcommand*{\backref}[1]{\ifx#1\relax \else Page #1 \fi}
\renewcommand*{\backrefalt}[4]{%
  \ifcase #1 \footnotesize{(Not cited.)}%
  \or        \footnotesize{(Cited on page~#2.)}%
  \else      \footnotesize{(Cited on pages~#2.)}%
  \fi
}

\def\namedlabel#1#2{\begingroup
    #2%
    \def\@currentlabel{#2}%
    \phantomsection\label{#1}\endgroup
}

\title[Weighted Laplacian-Eigenmap based Nonparametric Regression]{Adaptive and non-adaptive minimax rates for weighted Laplacian-Eigenmap based nonparametric regression}

\author{Zhaoyang Shi}
\author{Krishnakumar Balasubramanian}
\author{Wolfgang Polonik}

\address{\hspace{-0.16in}Department of Statistics, University of California, Davis.\newline Email: \texttt{zysshi@ucdavis.edu}, \texttt{ kbala@ucdavis.edu}, \texttt{wpolonik@ucdavis.edu}}
\usepackage[foot]{amsaddr}

\newtheorem{Theorem}{Theorem}[section]

\newtheorem{Proposition}{Proposition}[section]
\newtheorem{Remark}{Remark}[section]

\newtheorem{Lemma}{Lemma}[section]

\newcommand{\mbb}{\mathbb}
\newcommand{\mbf}{\mathbf}
\newcommand{\mcal}{\mathcal}

\DeclareMathOperator{\Var}{Var}

\allowdisplaybreaks

\begin{document}

\maketitle

\begin{abstract}
We show both adaptive and non-adaptive minimax rates of convergence for a family of weighted Laplacian-Eigenmap based nonparametric regression methods, when the true regression function belongs to a Sobolev space and the sampling density is bounded from above and below. The adaptation methodology is based on extensions of Lepski's method and is over both the smoothness parameter ($s\in\mbb{N}_{+}$) and the norm parameter ($M>0$) determining the constraints on the Sobolev space. Our results extend the non-adaptive result in \cite{green2021minimax}, established for a specific normalized graph Laplacian, to a wide class of weighted Laplacian matrices used in practice, including the unnormalized Laplacian and random walk Laplacian.
\end{abstract}

\section{Introduction}
Consider the following regression model,
\begin{align}\label{regressmodel}
    Y_{i}=f(X_{i})+\varepsilon_{i},\quad i=1,\ldots,n,
\end{align}
where $f:\mathcal{X}\to \mathbb{R}$ is the true regression function, $X_{i} \overset{\text{i.i.d.}}{\sim} g$, where $g$ is a density on $\mcal{X}\subset \mathbb{R}^d$, and $\varepsilon_{i}\overset{\text{i.i.d.}}{\sim} N(0,1)$ is the noise (independent of the $X_{i}$'s). The goal is to estimate the regression function $f$ given pairs of observations $(X_{1}, Y_{1}),\ldots, (X_{n}, Y_{n})$. Our main contribution in this work is to develop non-adaptive and adaptive estimators that achieve minimax optimal estimation rates, when $f$ lies in Sobolev spaces. 

The estimators we study are based on performing principal components regression using the estimated eigenfunctions of a family of weighted Graph Laplacian operators. To this end, various versions of Graph Laplacian matrices have been considered in the literature. Recently,~\cite{hoffmann2022spectral} proposed a unifying framework by describing a family of Graph Laplacian matrices, parametrized by $w\in\mathbb{R}^3$; see~\eqref{eq:weights} and~\eqref{weightedgraph} for details. This captures Laplacian matrices used widely in practice, including the normalized, unnormalized and the random walk Laplacian.

\cite{green2021minimax} analyzed principal components regression specifically using unnormalized graph Laplacian matrices constructed over $\epsilon$-graphs, and established non-adaptive minimax rates when $f$ lies in Sobolev spaces. In this paper, we first extend this result to the entire family of weighted Laplacian matrices from~\eqref{eq:weights} and~\eqref{weightedgraph}; Theorem \ref{firstorder}. These results are established by assuming a sampling density bounded from above and below and a true regression function belonging to a Sobolev space.

Note that technically, the weighted Laplacian matrices correspond to a family of weighted Sobolev spaces which all become equivalent under the above-mentioned boundedness assumption on the sampling density. However, the parameters of the corresponding Sobolev spaces, in particular the smoothness parameter ($s\in\mbb{N}_{+}$) and the norm parameter ($M>0$) 
determining the constraints on the Sobolev space, both change on $w$. 

While the minimax rate optimal non-adaptive estimator depends on the knowledge of the smoothness and norm parameters of the true regression function, these parameters are unknown in practice. Tuning parameters, such as $\epsilon$, the graph radius (or the bandwidth for the kernel) and $K$, the number of eigenvectors considered, require knowledge of the smoothness and the norm parameters. Hence, in order to apply the Laplacian-based regression methodology in practice, we develop an adaptive estimator, based on Lepski's method, and show in Theorem \ref{Lepski} that the developed estimator achieves minimax rates (up to $\log$ factors) without requiring the knowledge of either the smoothness or the norm parameters. 

The main technical contributions we make in this work towards establishing the aforementioned both adaptive and non-adaptive results include the following:
\begin{itemize}
    \item As a part of the proofs of our main results in Theorem~\ref{firstorder}, we rigorously prove the idea roughly outlined by \cite{hoffmann2022spectral} on showing the convergence of the discrete weighted graph Laplacian matrices to their continuum counterparts (in appropriately well-defined sense) by leveraging the concentration result established by \cite{gine2002rates} for kernel density estimators.
    \item We generalize the convergence property of the eigenvalues of the Laplacian matrices in \cite{calder2022improved} to the weighted Laplacian matrices by providing an analogous bound for the eigenvalues combined with Weyl's law.
    \item We formulate a simultaneous two-parameter Lepski's procedure and obtain the adaptive minimax rate (see Theorem~\ref{Lepski}) through deriving a high-order-moment-based concentration inequality of the weighted Sobolev semi-norm.
\end{itemize}

Our contributions not only highlight the significance of utilizing the weighted graph Laplacians for nonparametric regression but also establish a solid statistical foundation for this method, offering a robust framework that underpins the reliability and effectiveness of this approach.

\subsection{Related works}\label{sec:relatedworks}

Graph Laplacians are widely used in many data science problems for feature learning and spectral clustering~\cite{weiss1999segmentation,shiNormalizedCutsImage2000a,ngSpectralClusteringAnalysis2001,von2007tutorial}, extracting heat kernel signatures for shape analysis~\cite{sun2009concise, andreuxAnisotropicLaplaceBeltramiOperators2015, dunson2021spectral}, reinforcement learning~\cite{mahadevan2007proto,wu2018the} and dimensionality reduction \cite{belkin2003laplacian,coifman2006diffusion}, among other applications. There is an ever-growing literature on further applications of graph Laplacian in data science topic, and we also refer to \cite{belkin2006manifold, wang2015trend, chun2016eigenvector} for more discussions. 

As mentioned above, we consider the application of the weighted graph Laplacian for achieving minimax optimal rates in nonparametric regression. Other works focusing on this problem (including the semi-supervised setting) include  \cite{green2021minimaxsmoothing} and ~\cite{green2021minimax} using unnormalized Laplacian based on the Laplacian eigenmaps \cite{belkin2003laplacian}, \cite{bousquet2003measure} with Laplacian smoothing, \cite{rice1984bandwidth} adopting spectral series regression on the Sobolev spaces, \cite{trillos2022rates} applying the graph Poly-Laplacians (see Remark~\ref{rempoly} for specific comparison to this method) and \cite{hacquard2022topologically} using topological data analysis. We also refer to \cite{zhu2003semi}, \cite{zhou2011error}, \cite{lee2016spectral}, \cite{dicker2017kernel} and \cite{garcia2020maximum} for related analysis in the context of regression problems. 

In recent years, there has been a great deal of progress on obtaining theoretical rates of convergence in the context of Laplacian operator estimation and related eigenvalue and/or eigenfunction estimation. Early work on consistency of graph Laplacians focused on pointwise consistency results for $\epsilon$-graphs, see \cite{belkin2005towards,hein2005graphs, gine2006empirical,hein2007graph} and references therein for more details. For fixed neighborhood size $\epsilon$, \cite{von2008consistency} and \cite{rosasco2010learning} considered spectral convergence of graph Laplacians. Furthermore, \cite{trillos2018variational} established conditions on connectivity for the above spectral convergence with no specific error estimates. Later on, the convergence of Laplacian matrices to Laplacian operators has been considered where, for instance, unnormalized, random walk Laplacians and $k$-NN graph based Laplacians are considered; e.g. see ~\cite{shi2015convergence,garcia2020error,calder2022improved}. There, rates of convergences of Laplacian eigenvalues and eigenvectors to population counterparts with explicit error estimates are derived. Following the above literature, \cite{hoffmann2022spectral} developed a framework for extending the above convergence results to a general Laplacian family, the weighted Laplacians (see below), and presented some heuristic asymptotic analysis.

To the best of our knowledge, the only work that considers adaptivity in the context of Laplacian estimation is \cite{chazal2016data}. They use Lepski’s method for adaptive estimation of the unnormalized Laplace-Beltrami operators, focusing on bandwidth parameters. Also, they adopt a more flexible version of Lepski's method introduced in \cite{lacour2016minimal} that involves certain multiplicative coefficients introduced in the variance and
bias terms to develop the method. Therefore, their proof technique is to consider the trade-off between the bounds on the approximation error and the variance of Laplacian estimators. However, in this paper, we apply Lepski's method in the context of regression problem by using weighted Laplacians instead of just the unnormalized Laplacians (as in~\cite{chazal2016data}). Additionally, besides the bandwidth parameter, our method is also adaptive to the smoothness parameter and the norm parameter of the Sobolev space under consideration, i.e., in our work, we use Lepski's method for \emph{simultaneous} adaptation to the unknown parameters of the function class under consideration.

\section{Preliminaries}
In this section, we first describe the data-based weighted graph Laplacian matrices, and the corresponding nonparametric regression estimator. We then introduce the associated limiting operators and the weighted Sobolev spaces.

\subsection{Weighted graph Laplacian matrices}\label{egraph}
Given i.i.d data $X_{1},\ldots,X_{n}$ from a distribution $G$ on $\mcal{X}\subseteq \mbb{R}^{d}$ with the density $g$, consider a graph $G$ with vertex set $\{X_1,\ldots,X_n\}$ and adjacency matrix $\tilde W$ given by
\begin{align}\label{eq:weights}
    \tilde{w}_{i,j}^{\epsilon}:=\frac{1}{n\epsilon^{d}}\eta\left(\frac{\|X_i-X_j\|}{\epsilon}\right),\quad i,j=1,\ldots,n,
\end{align}
where $\| \cdot\|$ denotes the standard Euclidean norm. Here $\eta \ge 0$ is a kernel function with support $[0,1]$, and $\varepsilon$ is the bandwidth parameter.
In other words, $G$ is constructed by placing an edge $X_i\sim X_j$, when $\|X_i-X_{j}\|\le \epsilon,$ and this edge is given the weight $\tilde{w}_{i,j}^{\epsilon}.$ The term $(n\epsilon^{d})^{-1}$ is a convenient normalization factor. The degree matrix is then given by a diagonal matrix $\tilde{D}$ with the $i$-th diagonal element as 
\begin{align*}
    \tilde{d}_{i}:=\sum_{j=1}^{n}\tilde{w}_{i,j}^{\epsilon},\quad i=1,\ldots,n,
\end{align*}
which can also be thought of as the kernel density estimator (KDE) of the density $g$ at $X_i$. 

The weighted graph Laplacian matrices are a family of graph Laplacians consisting of various types of normalizations characterized by a parameter $w=(p,q,r)\in \mbb{R}^{3}$ constructed as follows. First define a re-weighted adjacency matrix $W$ with $(i,j)$-th element as 
\begin{align*}
    w_{i,j}^{\epsilon}:=\frac{\tilde{w}_{i,j}^{\epsilon}}{\tilde{d}_{i}^{1-\frac{q}{2}}\tilde{d}_{j}^{1-\frac{q}{2}}},\quad i,j=1,\ldots,n,
\end{align*}
so that the corresponding diagonal degree matrix $D$ as entries
\begin{align*}
    d_{i}:=\sum_{j=1}^{n}w_{i,j}^{\epsilon},\quad i=1,\ldots,n.
\end{align*}
Then, the weighted graph Laplacian after re-weighting is defined in \cite{hoffmann2022spectral} as follows: for a tuple $w=(p,q,r)\in\mbb{R}^{3}$,
\begin{equation}\label{weightedgraph}
        L_{w,n,\epsilon}:=
        \left\{
        \begin{aligned}
        &\frac{1}{\epsilon^{2}}D^{\frac{1-p}{q-1}}(D-W)D^{-\frac{r}{q-1}},\qquad\text{if}\ q\neq 1,\\
        &\frac{1}{\epsilon^{2}}(D-W),\qquad\qquad\qquad\quad\text{if}\ q=1,
        \end{aligned}
        \right.
\end{equation}
where $1/\epsilon^{2}$ is also a normalization factor. For $u\in \mbb{R}^{n}$, the $i$-th coordinate of the vector $L_{w,n,\epsilon}u$ is given by
\begin{align}\label{def:weightedL}
    (L_{w,n,\epsilon}u)_{i}=\frac{1}{\epsilon^{2}}\sum_{j=1}^{n}d_{i}^{\frac{1-p}{q-1}}w_{i,j}^{\epsilon}\left(d_{i}^{-\frac{r}{q-1}}u_i-d_{j}^{-\frac{r}{q-1}}u_j\right).
\end{align}
The above weighted graph Laplacian \eqref{weightedgraph} generalizes many commonly used graph Laplacian. For $(p,q,r)=(1,2,0)$, it recovers the unnormalized graph Laplacian $L_{u}$; if $(p,q,r)=(3/2,2,1/2)$, it gives the normalized graph Laplacian $L_{n}$; if $(p,q,r)=(2,2,0)$, it corresponds to a non-symmetric matrix but can be interpreted as a transition probability of a random walk on a graph denoted by $L_{r}$:
\begin{align*}
    L_{u}&:=D-W,\\
    L_{n}&:=D^{-1/2}(D-W)D^{1/2},\\
    L_{r}&:=D^{-1}(D-W),
\end{align*}

While the main focus on $\epsilon$-graphs, we highlight that the above formulation also caputures the limits of graph constructed based on the $k$-nearest neighbor graphs. In particular, when $(p,q,r)=(1,1-2/d,0)$, one can call the related normalization as the \textit{near $k$-NN} normalization; see~\cite{calder2022improved} and \cite{hoffmann2022spectral} for details.

Note that the weighted Laplacian matrix $L_{w,n,\epsilon}$ is actually not self-adjoint with respect to the Euclidean inner product $\langle\cdot,\cdot\rangle$ since it is in general not symmetric. However, it is self-adjoint with respect to the following weighted inner product $\langle \cdot,\cdot\rangle_{g^{p-r}}$:
\begin{equation*}
    \langle\cdot,\cdot\rangle_{g^{p-r}}:=
    \left\{
    \begin{aligned}
        &\langle\cdot,\cdot\rangle_{D^{\frac{p-1-r}{q-1}}}\qquad \text{if}\ q\neq 1,\\
        &\langle\cdot,\cdot\rangle\quad\quad\quad\quad\quad \text{if}\ q=1,
    \end{aligned}
    \right.
\end{equation*}
where for a given a symmetric matrix $A\in\mbb{R}^{n\times n}$ and vectors $u,v\in\mbb{R}^{n}$, define 
\begin{align*}
    \langle u,v \rangle_{A}:=u^{T}Av.
\end{align*}
We also define the normalized weighted inner product: $\langle\cdot,\cdot\rangle_{w,n}:=n^{-1}\langle\cdot,\cdot\rangle_{g^{p-r}}$ and the normalized Euclidean inner product: $\langle\cdot,\cdot\rangle_{n}:=n^{-1}\langle\cdot,\cdot\rangle$ and denote by $\|\cdot\|_{w,n}$ and $\|\cdot\|_{n}$ their respective corresponding norms. Here, our estimation results are measure in $\|\cdot\|_{w,n}$ and under our assumptions in Section \ref{sec:assump}, it can be shown to be equivalent to the classic norm $\|\cdot\|_{n}$.

\subsection{Weighted
Laplacian-Eigenmap based nonparametric regression}\label{modelandalg}

Following the ideas in \cite{belkin2003laplacian} and~\cite{green2021minimax}, we propose the following principal components regression with the weighted Laplacian eigenmaps (PCR-WLE) algorithm:
\begin{itemize}
    \item[(1)] For a given parameter $\epsilon>0$ and a kernel function $\eta$, construct the $\epsilon$-graph according to Section \ref{egraph}.
    \item[(2)] Construct the weighted Laplacian matrix given by \eqref{weightedgraph} and take its eigendecomposition $L_{w,n,\epsilon}=\sum_{i=1}^{n}\lambda_{i}v_{i}v_{i}^{T}$ with respect to $\langle\cdot,\cdot\rangle_{w,n}$, where $(\lambda_i,v_i)$ are the eigenpairs with eigenvalues $0=\lambda_1 \le \ldots\le \lambda_{n}$ in an ascending order and eigenvectors normalized  to satisfy $\|v_{i}\|_{w,n}=1$.
    \item[(3)] Project the response vector $Y=(Y_1,\ldots,Y_{n})^{T}$ onto the space spanned by the first $K$ eigenvectors, i.e., denote by $V_{K}\in \mbb{R}^{n\times K}$ the matrix with $j$-th column as $V_{K,j}=v_{j}$ for $j=1,\ldots,K$ and define
    \begin{align*}
        \hat{f}:=V_{K}V_{K}^{T}Y,
    \end{align*}
    as the estimator.
\end{itemize}

The entries of the vector $\hat f$ are the in-sample values of the estimator of the regression function $f$. \cite{green2021minimax} considered the special case of the above approach for the case when $(p, q, r) =
(1, 2, 0)$ corresponding to the unnormalized graph Laplacian. Here, we consider the entire family of graph Laplaicans for various choices of the parameters $(p, q, r)$, the generalization from~\cite{hoffmann2022spectral}.

\subsection{Weighted Laplacians and weighted Sobolev spaces}

\cite{hoffmann2022spectral} presented a heuristic framework for the convergence of the weighted graph Laplacian $L_{w,n,\epsilon}$ defined in~\eqref{weightedgraph} to the following weighted Laplace-Beltrami operators, in the large sample limit, in terms of the eigenvalues and eigenvectors/eigenfunctions:
\begin{equation}\label{weightedop}
        \left\{
        \begin{aligned}
        &\mcal{L}_{w}u:=-\frac{1}{2g^{p}}\text{div}\left(g^{q}\nabla\left(\frac{u}{g^{r}}\right) \right),\qquad\text{in}\ \mcal{X},\\
        &g^{q}\frac{\partial}{\partial n}\left(\frac{u}{g^{r}}\right)=0,\qquad\qquad\qquad\qquad\text{on}\ \partial \mcal{X}.
        \end{aligned}
        \right.
\end{equation}
Special cases of this convergence, including convergences of $L_u,L_{n},$ have been studied in \cite{shi2015convergence,garcia2020error,calder2022improved} as mentioned before in Section~\ref{sec:relatedworks}. 
Although our focus is not directly on the convergence of the weighted Laplacians but on regression problems, we digress slightly to make the following remark. The proof arguments developed in our paper, in the context of regression rates, can be applied to show the convergence of the weighted Laplacians, thereby rigorously proving the heuristic idea in \cite{hoffmann2022spectral}. This could be accomplished by using the concentration properties of kernel density estimators from \cite{gine2002rates} when the domain is has no  boundary, like we do in the context of regression rates. For domains with boundary is well-known that the convergence of the Laplacian matrices to the Laplacian operators is problematic at the boundary~ \cite{belkin2012toward}.

The weighted Laplacian operators are a generalization of the classical Laplacian operator with different values of $w=(p,q,r)$. Similar to the fact that the Laplacian operator is linked with the Sobolev space, the weighted Laplacian operators in \eqref{weightedop} share a close connection with the following so-called weighted Sobolev spaces; see~\cite{Triebel1983TheoryOF} for a general introduction. Define the weighted $L^{2}$ space for $\ell>0$ on $\mcal{X}$ with a density $g$ as
\begin{align*}
    L^{2}(\mcal{X},g^{\ell}):=\left\{u:\int_{\mcal{X}}|u(x)|^{2}g(x)^{\ell}dx<\infty\right\},
\end{align*}
with inner product
\begin{align*}
    \langle u,v \rangle_{g^{\ell}}:=\int_{\mcal{X}}u(x)v(x)g(x)^{\ell}dx.
\end{align*}
Then, for $w:=(p,q,r)\in\mbb{R}^{3}$ and $s\in \mbb{N}_{+}$, we define the weighted Sobolev space as:
\begin{align*}
H^{s}(\mcal{X},g):=\left\{\frac{u}{g^{r}}\in L^{2}(\mcal{X},g^{p+r}):\|u\|_{H^{s}(\mcal{X},g)}<\infty\right\},
\end{align*}
where the weighted Sobolev norm $\|u\|_{H^{s}(\mcal{X},g)}$ is 
\begin{align*}
    \|u\|_{H^{s}(\mcal{X},g)}^{2}:=\sum_{j=1}^{s}|u|_{H^{j}(\mcal{X},g)}^2+\left\|\frac{u}{g^{r}}\right\|_{L^{2}(\mcal{X},g^{p+r})}^{2},
\end{align*}
with the $j$-th order semi-norm $|\cdot|_{H^{j}(\mcal{X},g)}$ defined as $$|u|_{H^{j}(\mcal{X},g)}:=\sum_{|\alpha|=j}\left\|D^{\alpha}\left(ug^{-r}\right)\right\|_{L^{2}(\mcal{X},g^{p})}$$ and using multi-index notation with $x=(x^{(1)},\ldots,x^{(d)})\in \mbb{R}^{d},$ we have that $D^{\alpha}f(x):=\partial^{|\alpha|}f/\partial (x^{(1)})^{\alpha_1}\ldots \partial (x^{(d)})^{\alpha_d}$ and $|\alpha|=\alpha_1+\ldots+\alpha_{d}$. When $g$ is uniform or $r=0$ and $g$ is bounded from above and below, the weighted Sobolev space $H^{s}(\mcal{X},g)$ becomes (or is equivalent to) the classic Sobolev space $H^{s}(\mcal{X})$. However, when $f/g^{r}$ is  $s$-times differentiable but $f$ is not, the weighted Sobolev space differs from the classic Sobolev space. See \cite{evans2022partial} for more details regarding Sobolev spaces. For $M>0$, the class of all functions $u$ such that $\|u\|_{H^{s}(\mcal{X},g)}\le M$ is a weighted Sobolev ball $H^{s}(\mcal{X},g;M)$ of radius $M.$ 

Furthermore, we say a function $u\in H^{s}(\mcal{X},g)$ belongs to the zero-trace weighted Sobolev space $H_{0}^{s}(\mcal{X},g)$ if there exists a sequence $u_1g^{-r},\ldots,u_{m}g^{-r}$ of $C_{c}^{\infty}(\mcal{X})$ functions such that 
\begin{align*}
    \underset{m\rightarrow\infty}{\lim}\|u_{m}-u\|_{H^{s}(\mcal{X},g)}=0,
\end{align*}
where $C_{c}^{\infty}(\mcal{X})$ stands for the $C^{\infty}$ functions with compact support contained in $\mcal{X}$. 

Similar to the weighted Laplacian matrix $L_{w,n,\epsilon}$, the weighted Laplacian operators \eqref{weightedop} are self-adjoint with respect to the following weighted inner product (\cite{hoffmann2022spectral}):
\begin{align*}
    \langle u, v\rangle_{g^{p-r}}:=\int_{\mcal{X}}u(x)v(x)g^{p-r}(x)dx.
\end{align*}
Note the following connection between the weighted norms and inner products:
\begin{align*}
    \left\|\frac{u}{g^{r}}\right\|_{L^{2}(\mcal{X},g^{p+r})}^{2}=\|u\|_{L^{2}(\mcal{X},g^{p-r})}^{2}=\langle u,u\rangle_{g^{p-r}}.
\end{align*}

A simple example showing the dependency of the choice $M$ on $p,q,r$ is as follows. Suppose that $u/g^{r}$ is the constant function equal to $1$, say, and take $s=1$. Then, we have
\begin{align*}
    \|u\|_{H^{1}(\mcal{X},g)}^{2}=\int_{\mcal{X}}g(x)^{p+r}dx.
\end{align*}
Clearly, the power $p+r$ of the density function $g$ determines the size of the weighted Sobolev ball, and thus $M$. In other words, say for example, assuming $g\ge 1$ for simplicity, larger configurations of $p+r$ will result in large weighted Sobolev norm, thus requiring a large norm parameter $M$. For generic $u/g^{r}$, the situation is more intricate and depends on the geomtry of $u$ and $g$ and choices of $p+r$.

\section{Main results}

We now present our main results on adaptive and non-adaptive rates for estimating the regression function $f$ as in \eqref{regressmodel} under some smoothness assumptions. Before that, we recall that the minimax estimation error over $H^{s}(\mcal{X};M)$, a standard Sobolev ball of radius $M$, is given by  
\begin{align*}
    \underset{\hat{f}}{\inf}\underset{f\in H^{s}(\mcal{X};M)}{\sup}\|\hat{f}-f\|_{n}^{2}\asymp M^{2}(M^{2}n)^{-\frac{2s}{2s+d}},
\end{align*}
with high probability~\cite{gyorfi2002distribution,wasserman2006all,tsybakov2008ntroduction}. Moreover, there are other methods that can achieve the above minimax rate such as kernel smoothing, local polynomial regression, thin-plate splines, etc. In this context,~\cite{green2021minimax} showed that PCR-WLE  method with the unnormalized Laplacian\footnote{This procedure is refered to as PCR-LE in~\cite{green2021minimax}.} $L_{u}$  achieves the minimax rate, provided that $n^{-1/2}\lesssim M\lesssim n^{s/d}$ under appropriate assumptions, where for two real-valued quantities, $A,B$, the notation $A \lesssim B$ means that there exists a constant $C > 0$ not depending on $f$, $M$ or $n$ such that $A\le CB$ and $A\asymp B$ stands for $A\lesssim B$ and $B\lesssim A$.

\subsection{Assumptions}\label{sec:assump}

We now list the major assumptions that are needed for our theoretical results.

\begin{itemize}
    \item[\namedlabel{a1}{(A1)}] The distribution $G$ is supported on $\mcal{X}$, which is an open, connected, and bounded subset of $\mbb{R}^{d}$ with Lipschitz boundary.
    \item[\namedlabel{a2}{(A2)}] The distribution $G$ has a density $g$ on $\mcal{X}$ such that
    \begin{align*}
        0<g_{\min}\le g(x)\le g_{\max}<\infty,\ \text{for all}\ x\in \mcal{X},
    \end{align*}
    for some $g_{min},g_{\max}>0$. Additionally, $g$ is Lipschitz on $\mcal{X}$ with Lipschitz constant $L_{g}>0$.
    \item[\namedlabel{a3}{(A3)}] The kernel $\eta$ is a non-negative, monotonically non-decreasing function supported on the interval $[0,1]$ and its restriction on $[0,1]$ is Lipschitz and for convenience, we assume $\eta(1/2)>0$ and define
    \begin{align*}
    \sigma_{0}:=\int_{\mbb{R}^{m}}\eta(\|x\|)dx,\quad \sigma_{1}:=\frac{1}{d}\int_{\mbb{R}^{m}}\|y\|^{2}\eta(\|y\|)dy.
    \end{align*}
    Without loss of generality, we will assume $\sigma_0=1$ from now on.

    \item[\namedlabel{a4}{(A4)}] The kernel $\eta$ satisfies a kernel VC-type condition as follows. Let
    \begin{align*}
    \mathscr{K}:=\left\{y\rightarrow \eta\left(\frac{x-y}{\epsilon}\right): \epsilon>0,x\in\mbb{R}\right\}
    \end{align*}
    be the collection of kernel functions indexed by $x$ and $\epsilon$. For a density $\rho$, let the $L^{2}(\mcal{X},\rho)$-covering number $N(\epsilon, {\mathscr{K}}, \|\cdot\|_{L^{2}(\mcal{X},\rho)})$ of ${\mathscr{K}}$ be the smallest number of $L^{2}(\mcal{X},\rho)$-balls of radius $\epsilon$ needed to cover $\mathscr{K}$.
    With that we say that $\eta$ satisfies the kernel VC-type condition if there exist constants $A,\nu>0$ such that
    \begin{align}\label{VC}
    \sup_{\rho}~N(\zeta,\mathscr{K},\|\cdot\|_{L^{2}(\mcal{X},\rho)})\le \left(\frac{A}{\zeta}\right)^{\nu},
    \end{align}
    See Remark \ref{kernelvc} for some examples.
\end{itemize}

Assumptions \ref{a1} and \ref{a2} are mild assumption on the density function, which are also made in~\cite{green2021minimax}. In particular \ref{a2} is important for us, as it gives us the norm equivalence between the various families of weighted Sobolev spaces. Assumption \ref{a3} is a standard normalization condition made on the smoothing kernel, also made in~\cite{green2021minimax}. Assumption \ref{a4} is \textit{not} used in~\cite{green2021minimax}. It is used here because the general family of weighted Laplacian matrices that we work with involve kernel density estimation normalization, with which the normalization in \eqref{weightedgraph} will not tend to either infinity or zero. Also note that in general condition $(\ref{VC})$ involves the $L^{2}(\mcal{X},\rho)$-norm of an envelope function $\eta_0$ for $\mathscr{K}$, i.e. of a function $\eta_0 \le h$ for all $h \in \mathscr{K}$. Since, by our assumptions, $\eta$ is bounded, we can use the maximum of $\eta$ as an envelope, for which the $L^{2}(\mcal{X},\rho)$-norm obviously does not depend on $\rho$ and can thus be absorbed by the constant $A$. 

\subsection{Non-adaptive rates}\label{mainthm}
In the following, we present the non-adaptive minimax optimal rate of convergence of the PCR-WLE estimator in Section \ref{modelandalg} for $s=1$ and $s>1$ separately. These rates are non-adaptive as the choice of $K$ and $\epsilon$ depends on unknown problem parameters, the smoothness parameter $s$ and the norm parameter $M$. 

\begin{Theorem}[Non-adaptive minimax rate of PCR-WLE algorithm]\label{firstorder}
    Assume \ref{a1}-\ref{a4}.
    \begin{itemize}
    \item [(a)] For $s\in \mbb{N}_{+}\backslash\{1\}$, assume $f\in H_{0}^{s}(\mcal{X},g;M)$, $f\in H^{1}(\mcal{X},g;M)$ and $g\in C^{s-1}(\mcal{X})$. Suppose there exist constants $c_0,C_0>0$ such that 
    \begin{align}\label{coneps}
    c_0\Bigg(\left(\frac{\log n}{n}\right)^{\frac{1}{d}}\vee &\;(M^{2}n)^{-\frac{1}{2(s-1)+d}}\Bigg)\le \,\epsilon \le C_0 K^{-\frac{1}{d}}, \nonumber\\[3pt]
        \text{\rm and}\hspace*{2cm}\nonumber\\[-10pt]%
        &\hspace*{0.5cm}\sqrt{\frac{|\log \epsilon |}{n\epsilon^{d}}}\rightarrow 0,
    \end{align}
    where 
    \begin{align}\label{conK}
        K=\min\left\{\lfloor (M^{2}n)^{\frac{d}{2s+d}}\rfloor \vee 1,n\right\}.
    \end{align}
    Then, there exist constants $c,C>0$ not depending on $f,M$ or $n$ such that for $n$ large enough and any $0<\delta<1$, we have:
    \begin{align*}
        \|\hat{f}-f\|_{w,n}^{2}\le C\big\{\big(\delta^{-1}M^2(M^2n)^{-\frac{2s}{2s+d}}\wedge 1\big) \vee {n^{-1}}\big\},
    \end{align*}
with probability at least $1-\delta-Cne^{-cn\epsilon^{d}}-e^{-K}$.
    \item [(b)] For $s=1$, assume $f\in H^{1}(\mcal{X},g;M)$. Suppose there exist constants $c_0,C_0>0$ such that 
    \begin{align*}
         c_0\left(\frac{\log n}{n}\right)^{\frac{1}{d}}&\le \epsilon \le C_0  K^{-\frac{1}{d}},
    \end{align*}
    and \eqref{coneps}, where $K$ is given in \eqref{conK} for $s=1$. Then, the assertion in part (a) also holds for $s=1$.
\end{itemize}
\end{Theorem}

\begin{Remark}
    Notably, the above theorems do not require the assumption that $s>d/2$. As we mentioned before in Section \ref{modelandalg}, this condition is commonly appeared in the literature as in the sub-critical regime, i.e., $s\le d/2$, the (weighted) Sobolev space $H^{s}$ is not a Reproducing Kernel Hilbert Space (RKHS) and cannot be continuously embedded into the space of continuous functions $C^{0}(\mcal{X})$. Theorem \ref{firstorder} highlights the point that PCR-WLE algorithm obtains the minimax optimal rate when $n^{-1/2}\lesssim M\lesssim n^{s/d}$ and the error is measured by the weighted empirical norm $\|\cdot\|_{w,n}$. 
\end{Remark}

\begin{Remark}\label{kernelvc}
    The kernel VC-type condition was first proposed in \cite{gine2002rates}. A simple sufficient condition for this condition to hold is that $\eta$ is of bounded variation; see \cite{nolan1987Ustatistics} or \cite{gine2021mathematical}. Clearly, many  common kernels are of this type, including Gaussian, Epanechnikov and cosine kernels.
\end{Remark}

\begin{Remark}
    For practical consideration, there are two tuning parameters: the graph radius (the bandwidth for the kernel $\eta$) $\epsilon$ and the number of eigenvalues $K$. The lower bound for $\epsilon$ makes sure that with this smallest radius, the resulting weighted graph will still be connected with high probability and the upper bound for $\epsilon$ ensures the eigenvalue of the weighted graph Laplacian \eqref{weightedgraph} to be of the same order as its  continuum version, the eigenvalue of the weighted Laplacian operator \eqref{weightedop} (Weyl's law). The asymptotic assumption on $\epsilon$ is from the concentration of the KDE. The condition on $K$ is set to trade-off bias and variance. Both $\epsilon$ and $K$ depend on the true smoothness parameter $s\in\mbb{N}_{+}$. 
\end{Remark}

\subsection{Adaptive rates via Lepski's method}\label{seclepski}

Despite the minimax optimality of the PCR-WLE algorithm shown in Section \ref{mainthm}, the main practical difficulties are the choice of several tuning parameters including the bandwidth parameter (or the graph radius) $\epsilon$ and the number of eigenvalues $K$, because optimal choices depend on the \textit{unknown} true smoothness parameter $s$ of the regression function $f$ in the model \ref{regressmodel}. Moreover, $K$ also relies on the bound of the weighted Sobolev norm $M$. This naturally brings about the issue of adaptation, which we address using  Lepski's method. Note that, as we are concerned with in-sample estimation error, other techniques like cross-validation are not directly applicable to set the tuning parameters. 

Since its introduction in \cite{lepskii1991problem}, Lepski's method has been widely used for adaptive estimation and testing in various statistical contexts; e.g. see \cite{birge2001alternative,chichignoud2016practical,bellec2018slope,balasubramanian2021optimality,lacour2016minimal}. In the following, we consider Lepski's method on the product space of the smoothness parameter $s\in\mbb{N}_{+}$ and the constraint on the weighted Sobolev norm $M\in \mbb{R}_{+}$.

Recall that $s$ and $M$ denote the true smoothness parameter and the norm parameter, respectively for the weighted Sobolev norm of $f$. Here, we actually take $M$ as the minimum over all bounds of the weighted Sobolev norm. We start by  picking $s_{\min},s_{\max} \in \mbb{N}_{+}$; here we can set $s_{\min}=1$ under no availability of further information\footnote{If there is additional information, like $s>10$, one can pick $s_{\min}=10$. Hence, we present out result with a generic $s_{\min}$.} regarding the knowledge of $s$. The goal is that $s_{\max}$ is large enough that  $s\in\mbb{N}_{+}$ satisfies $s\in[s_{\min},s_{\max}]$. Similarly, we pick $M_{\min},M_{\max}$ satisfying $0<M_{\min}<M_{\max}<\infty$, where $M_{\min}$ and $M_{\max}$ are small and large enough respectively such that $M\in [M_{\min},M_{\max}]$. Next, define the grid $\mcal{B}\times \mcal{D}:=\{(s_j,M_j)\}_{j=1}^{N_l}$ given by:
\begin{align}\label{eq:sinterval}
\begin{aligned}
    \mcal{B}\coloneqq[s_{\min},s_{\max}]\cap \mbb{N}_{+} =\{s_{\min}=:s_1<s_2<\ldots<s_{N_{l}}:=s_{\max}\},
\end{aligned}
\end{align}
and 
\begin{align*}
    \mcal{D}\coloneqq[M_{\min},M_{\max}]=\{M_{\min}=:M_1<M_2<\ldots<M_{N_{l}}:=M_{\max}\},
\end{align*}
where $N_{l}\asymp \log n$.

For any pair $(\tilde{s},\tilde{M})$ in the above grid, let $\hat{f}_{\tilde{s},\tilde{M}}$ be the PCR-WLE estimator in Section \ref{modelandalg} corresponding to the parameters $\tilde{s}$ and $\tilde{M}$. We define the Lepski's estimator as
\begin{align*}
    \hat{f}_{\adapt}:=\hat{f}_{\hat{s},\hat{M}},
\end{align*}
where $\hat{s}$ is given by
\begin{align*}
    \hat{s}:=\max \{\tilde{s}\in \mcal{B}: \|\hat{f}_{\tilde{s},\tilde{M}}-\hat{f}_{\tilde{s}',\tilde{M}'}\|_{w,n}  \le c_0 \tilde{M}'((\tilde{M}'^2 n/\log n)^{-\frac{\tilde{s}'}{2\tilde{s}'+d}}\ ,\forall \tilde{s}'\le \tilde{s}, \tilde{s}'\in\mcal{B}\},
\end{align*}
and $\hat{M}$ is the corresponding couple of $\hat{s}$ in the grid, where $c_0>0$ is some finite constant. Here, we formulate the above simultaneous Lepski's method by coupling the smoothness parameter and the norm parameter and only maximize through the smoothness parameter instead of dealing with a joint maximization, which is not needed for our purpose of showing the adaptive minimax rate in the following result as our focus is its convergence rate in $n$.

The following result presents a near minimax optimal rate of convergence of the Lepski's estimator $\hat{f}_\adapt$ up to a logarithmic factor in $n$.

\begin{Theorem}\label{Lepski}
    Assume \ref{a1}-\ref{a4} and $g\in C^{s-1}(\mcal{X})$. Also, assume $f\in H^{1}(\mcal{X},g;M) \cap H_{0}^{s}(\mcal{X},g;M)$ and  $f_g:=f/g^r$ is $M$-Lipschitz, i.e., $\|f_{g}(x)-f_{g}(x')\|\le M\|x-x'\|$ for any $x,x'\in \mcal{X}$. Furthermore, assume that (for large enough $n$) we have $s \in [s_{\min},s_{\max}]$ and $M \in [M_{\min},M_{\max}].$ Then, under the minimax optimal setting in Theorem \ref{firstorder} for $M$, i.e., $n^{-1/2}\lesssim M\lesssim n^{s/d}$, the estimator $\hat{f}_{\adapt}$ satisfies: For $n$ large enough and any $\delta\in (0,1)$, there exists some constant $C>0$ such that 
    \begin{align*}
    \|\hat{f}_{\adapt} -f\|_{w,n}^{2}\le C\delta^{-1}M^2(M^2n/\log n)^{-\frac{2s}{2s+d}},
    \end{align*}
    with probability at least 
    \begin{align*}
    1-\delta \log^{-\frac{2s}{(2s+d)}}n -Cn e^{-Cn\epsilon^{d}}\log^2 n -16C c_0^{-4} n^{-1}\log^{2-\frac{2s_{\min}}{(2s_{\min}+d)}} n-e^{-\lfloor M_{\min}^2 n \rfloor^{\frac{d}{(2s+d)}}}\log^2 n.
    \end{align*}
\end{Theorem}

\begin{Remark}\label{rempoly}
\cite{trillos2022rates} proposed a graph poly-Laplacian regularization approach, where integer powers of the Laplacian matrices are used as regularization in a least-squares context. They showed that the proposed method achieves rate of convergence of order $n^{-s/(d+4s)}$. While the rate is not optimal, in comparison to the~\cite{green2021minimax} their estimator does not require the knowledge of the norm parameter $M$ to achieve the derived rate (although they require the knowledge of $s$). In comparison to both the above works, our result in Theorem~\ref{Lepski} achieves the optimal rate, up to $\log$ factors, without requiring the knowledge of either $s$ or $M$.
\end{Remark}

\begin{Remark}
    As a part of our proof, a better concentration inequality for the non-adaptive PCR-WLE estimator $\hat{f}$ is required compared to Theorem \ref{firstorder}, for which the assumption that $f_g$ is Lipschitz is required. As also discussed in  \cite[see below Theorem 1]{green2021minimax}, it remains open whether a weaker assumption or even the weighted Sobolev condition $\|\nabla f_g\|_{L^{2}}<\infty$ alone might be sufficient establish the required concentration result for developing adaptive procedures.
\end{Remark}

\section{PROOF}

\subsection{Proof of Theorem \ref{firstorder}}

In this section, we will prove both Theorem \ref{firstorder} for $s=1$ and $s>1$ together. We first present and prove some auxiliary lemmas. We will denote by $B_x(r)$ a closed Euclidean ball with midpoint $x$ and radius $r\ge 0$.

Define the weighted Sobolev seminorm $\langle L_{w,\epsilon}f,f\rangle_{g^{p-r}}$ given by the following non-local operator:
\begin{align*}
    L_{w,\epsilon}f(x):=\frac{1}{\epsilon^{d+2}}\int_{\mcal{X}}g(x)^{1-p}\frac{\eta\big(\frac{\|x-z\|}{\epsilon}\big)}{g(x)^{1-q/2}g(z)^{1-q/2}}(g(x)^{-r}f(x)-g(z)^{-r}f(z))g(z)dz,
\end{align*}
where according to \eqref{def:weightedL}, $L_{w,\epsilon}$ can be viewed as a population counterpart of the discrete graph weighted Laplacian $L_{w,n,\epsilon}$. As in \cite{green2021minimax}, we also call it a `non-local' version. Note that the above non-local weighted Sobolev seminorm and non-local operator generalize the definitions in \cite{green2021minimax} as the latter belong to a special case when $(p,q,r)=(1,2,0)$. The following Lemmas \ref{lemma61}-\ref{lemma62} therefore extend their counterparts in \cite{green2021minimax} to the weighted Laplacians and the weighted Sobolev seminorm. \textcolor{black}{Note that in our proofs, we also highlight and fix several important typos and errors that appeared in \cite{green2021minimax}. Despite the errors, the final results in \cite{green2021minimax} remain true}.

\begin{Lemma}\label{lemma61}
    For $f\in H^{1}(\mcal{X},g;M)$, we have
    \begin{align*}
        \langle L_{w,\epsilon}f,f\rangle_{g^{p-r}}\lesssim M^{2}.
    \end{align*}
\end{Lemma}

\begin{proof}[Proof of Lemma \ref{lemma61}]
    Following the idea of \cite[Proof of Lemma 1]{green2021minimaxsmoothing}, take $\Omega$ as an arbitrary bounded open set such that $B_{x}(c_0)\subseteq \Omega$ for all $x\in\mcal{X}$ for some $c_0>0$ and we can assume that $f\in H^{1}(\Omega,g)$ and $\|f\|_{H^{1}(\Omega,g)}\lesssim \|f\|_{H^{1}(\mcal{X},g)}$ without loss of generality due to the existence of an extension operator $E:H^{1}(\mcal{X},g)\rightarrow H^{1}(\Omega,g)$ such that $Ef$ satisfies these properties, see Theorem 1 in Chapter 5.4 in \cite{evans2022partial}. Also, since $C^{\infty}(\Omega)$ is dense in $H^{1}(\Omega,g)$ and the integral in Lemma \ref{lemma61} is continuous in $H^{1}(\Omega,g)$, we can assume $f_{g}:=f/g^{r}\in C^{\infty}(\Omega)$ so that
\begin{align*}
    f_{g}(x')-f_{g}(x)=\int_{0}^{1}\nabla f_{g}(x+t(x'-x))^{T}(x'-x)dt.
\end{align*}
Then, we have by symmetry in the first step:
\begin{align}\label{s1decomp}
    &2\langle L_{w,\epsilon}f,f\rangle_{g^{p-r}}
    \nonumber\\&=\frac{1}{\epsilon^{d+2}}\int_{\mcal{X}}\int_{\mcal{X}}\frac{\eta\left(\frac{\|x-y\|}{\epsilon}\right)}{g(x)^{1-q/2}g(y)^{1-q/2}}\left|\frac{f(x)}{g(x)^{r}}-\frac{f(y)}{g(y)^{r}}\right|^{2}g(x)g(y)dxdy\nonumber\\&=\frac{1}{\epsilon^{d+2}}\int_{\mcal{X}}\int_{\mcal{X}}\frac{\eta\left(\frac{\|x-y\|}{\epsilon}\right)}{g(x)^{1-q/2}g(y)^{1-q/2}}\left(\int_{0}^{1}\nabla f_{g}(y+t(x-y))^{T}(x-y)dt\right)^{2}g(x)g(y)dxdy\nonumber\\&\le \frac{1}{\epsilon^{d+2}}\int_{\mcal{X}}\int_{\mcal{X}}\int_{0}^{1}\frac{\eta\left(\frac{\|x-y\|}{\epsilon}\right)}{g(x)^{1-q/2}g(y)^{1-q/2}}\left(\nabla f_{g}(y+t(x-y))^{T}(x-y)\right)^{2}g(x)g(y)dtdxdy\nonumber\\&\le \int_{\mcal{X}}\int_{B_{\mbf{0}}(1)}\int_{0}^{1}\left(\nabla f_{g}(y+\epsilon tz)^{T}z\right)^{2}\frac{\eta(\|z\|)}{g(y+\epsilon z)^{1-q/2}g(y)^{1-q/2}}g(y+\epsilon z)g(y)dtdzdy,\nonumber\\&\qquad \text{with}\ (x-y)/\epsilon=z\nonumber\\&\lesssim  \int_{\mcal{X}}\int_{B_{\mbf{0}}(1)}\int_{0}^{1}\left(\nabla f_{g}(y+\epsilon tz)^{T}z\right)^{2}\eta(\|z\|)g(y+\epsilon tz)^{q}dtdzdy\nonumber\\&\le \int_{\Omega}\int_{B_{\mbf{0}}(1)}\int_{0}^{1}\left(\nabla f_{g}(\tilde{y})^{T}z\right)^{2}\eta(\|z\|)g(\tilde{y})^{q}dtdzd\tilde{y},\quad \tilde{y}=y+\epsilon tz\in \Omega.
\end{align}
Since we have $\left(\nabla f_{g}(\tilde{y})^{T}z\right)^{2}=\left(\sum_{i=1}^{d}(\nabla f_{g}(\tilde{y}))^{(i)}z^{(i)}\right)^{2}$ and $\eta(\|z\|)$ is invariant with respect to the rotation, it yields that
\begin{align}\label{s1decomppart}
    \int_{B_{\mbf{0}}(1)}\left(\nabla f_{g}(\tilde{y})^{T}z\right)^{2}\eta(\|z\|)dz&=\sum_{i,j=1}^{d}(\nabla f_{g}(\tilde{y}))^{(i)}(\nabla f_{g}(\tilde{y}))^{(j)}\int_{B_{\mbf{0}}(1)}z^{(i)}z^{(j)}\eta(\|z\|)dz\nonumber\\&=\sum_{i=1}^{d}\left((\nabla f_{g}(\tilde{y}))^{(i)}\right)^{2}\int_{B_{\mbf{0}}(1)}\left(z^{(i)}\right)^{2}\eta(\|z\|)dz\nonumber\\&=\sigma_{1}\left\|\nabla\left(\frac{f(\tilde{y})}{g(\tilde{y})^{r}}\right)\right\|^{2}.
\end{align}
Plugging \eqref{s1decomppart} in \eqref{s1decomp}, we conclude
\begin{align*}
    2\langle L_{w,\epsilon}f,f\rangle_{g^{p-r}}\lesssim \sigma_1M^{2}.
\end{align*}
This finishes the proof.
\end{proof}

Note that the proof of Lemma \ref{lemma61} also utilized the heuristic arguments given in \cite{hoffmann2022spectral} while we provide a rigorous proof here.

\begin{Lemma}\label{lemma6}
Suppose $f_{g}\in L^{2}(\mcal{U},g^{p+r};M)$ for a Borel set $\mcal{U}\subseteq \mcal{X}$. Then, there exists a constant $C$ which does not depend on $f$ or $M$ such that
\begin{align*}
    \|L_{w,\epsilon}f\|_{L^{2}(\mcal{U},g^{p+r})}\le \frac{C}{\epsilon^{2}}\|f_{g}\|_{L^{2}(\mcal{U},g^{p+r})}.
\end{align*} 
\end{Lemma}

\begin{proof}
By Cauchy-Schwarz inequality, we have
\begin{align*}
    |L_{w,\epsilon}f(x)|^2&=\frac{1}{\epsilon^{2(d+2)}}\left(\int_{\mcal{U}}g(x)^{1-p}\frac{\eta\left(\frac{\|x-z\|}{\epsilon}\right)}{g(x)^{1-q/2}g(z)^{1-q/2}}(g(x)^{-r}f(x)-g(z)^{-r}f(z))g(z)dz\right)^{2}\\&\lesssim \frac{1}{\epsilon^{2(d+2)}}g(x)^{2(1-p)}\int_{\mcal{U}}\frac{\eta\left(\frac{\|x-z\|}{\epsilon}\right)}{g(x)^{1-q/2}g(z)^{1-q/2}}(f_{g}(x)-f_{g}(z))^2dz\cdot\int_{\mcal{X}}\frac{\eta\left(\frac{\|x-z\|}{\epsilon}\right)}{g(x)^{1-q/2}g(z)^{1-q/2}}dz\\&\lesssim \frac{2\sigma_0}{\epsilon^{4+d}}g(x)^{2(q-p)}\int_{\mcal{U}}\eta\left(\frac{\|x-z\|}{\epsilon}\right)(|f_{g}(x)|^2+|f_{g}(z)|^2)dz.
\end{align*}
Then, we have 
\begin{align*}
    \|L_{w,\epsilon}f\|_{L^{2}(\mcal{U},g^{p+r})}^{2}&=\int_{\mcal{U}}g^{p-r}(x)|L_{w,\epsilon}f(x)|^{2}dx\\&\lesssim  \frac{2}{\epsilon^{4+d}}\int_{\mcal{U}}\int_{\mcal{U}}g(x)^{2(q-p)+p-r}(x)\eta\left(\frac{\|x-z\|}{\epsilon}\right)(|f_{g}(x)|^2+|f_{g}(z)|^2)dzdx\\&\lesssim \frac{2}{\epsilon^{4+d}}\int_{\mcal{U}}\int_{\mcal{U}}\eta\left(\frac{\|x-z\|}{\epsilon}\right)(|f_{g}(x)|^2+|f_{g}(z)|^2)dzdx\\&\lesssim \frac{4}{\epsilon^{4+d}}\int_{\mcal{U}}\int_{\mcal{U}}\eta\left(\frac{\|x-z\|}{\epsilon}\right)|f_{g}(x)|^2dzdx\\&\le \frac{4}{\epsilon^{4}}\int_{\mcal{U}}|g(x)^{p+r}(x)f_{g}(x)|^2dx\\&\lesssim \frac{4}{\epsilon^{4}}\|f_{g}\|_{L^{2}(U,g^{p+r})}^{2}
\end{align*}
\end{proof}

\begin{Lemma}\label{lemma64}
    Suppose $f_{g}\in L^{2}(\mcal{U},g^{p+r};M)$ for a Borel set $\mcal{U}\subseteq \mcal{X}$. Then, there exists a constant $C>0$ such that
    \begin{align*}
        E_{w,\epsilon}(f;\mcal{U})\le \frac{C}{\epsilon^{2}}\|f_{g}\|_{L^{2}(\mcal{U},g^{p+r})}^{2},
    \end{align*}
    where we define the Dirichlet energy for the set $\mcal{U}$ as
\begin{align*}
    E_{w,\epsilon}(f,\mcal{U})&:=\frac{1}{\epsilon^{d+2}}\int_{\mcal{U}}\int_{\mcal{U}}(g(x)^{-r}f(x)-g(z)^{-r}f(z))^2\frac{\eta\left(\frac{\|x-z\|}{\epsilon}\right)}{g(x)^{1-q/2}g(z)^{1-q/2}}g(x)g(z)dxdz.
\end{align*}
\end{Lemma}

\begin{proof}
Note that
\begin{align*}
    E_{w,\epsilon}(f;\mcal{U})&=\frac{1}{\epsilon^{d+2}}\int_{\mcal{U}}\int_{\mcal{U}}(g(x)^{-r}f(x)-g(z)^{-r}f(z))^2\frac{\eta\left(\frac{\|x-z\|}{\epsilon}\right)}{g(x)^{1-q/2}g(z)^{1-q/2}}g(x)g(z)dxdz\\&\le \frac{2}{\epsilon^{d+2}}\int_{\mcal{U}}\int_{\mcal{U}}(|g(x)^{-r}f(x)|^2+|g(z)^{-r}f(z)|^2)\frac{\eta\left(\frac{\|x-z\|}{\epsilon}\right)}{g(x)^{1-q/2}g(z)^{1-q/2}}g(x)g(z)dxdz\\&= \frac{4}{\epsilon^{d+2}}\int_{\mcal{U}}\int_{\mcal{U}}|g(x)^{-r}f(x)|^2\eta\left(\frac{\|x-z\|}{\epsilon}\right)g(x)^{q/2}g(z)^{q/2}dxdz\\&\lesssim \frac{4}{\epsilon^{d+2}}\int_{\mcal{U}}\int_{\mcal{U}}|g(x)^{-r}f(x)|^2\eta\left(\frac{\|x-z\|}{\epsilon}\right)g(x)^{p+r}dxdz\\&\lesssim \frac{4}{\epsilon^{2}}\int_{\mcal{U}}|g(x)^{-r}f(x)|^2g(x)^{p+r}dx.
\end{align*}
\end{proof}

We denote by $\mcal{X}_{t\epsilon}$ a subset of $\mcal{X}$ such that for any $x\in \mcal{X}_{t\epsilon}$, $B_{x}(t\epsilon)\in \mcal{X}$ consisting of points sufficiently far away from the boundary and $\partial_{t\epsilon} \mcal{X}$ by its complement within $\mcal{X}$ consisting of points close enough to the boundary.

\begin{Lemma}\label{lemma62interior}
    For $f\in H^{1}(\mcal{X},g;M)\cap H_{0}^{s}(\mcal{X},g;M)$ with $s\in\mbb{N}_{+}$ and $g\in C^{s-1}(\mcal{X})$, there exist constants $C_1,C_2>0$ such that
\begin{itemize}
    \item[(1)] If $s$ is odd, then we have with $t=(s-1)/2$:
    \begin{align*}
        \|L_{w,\epsilon}^{t}f-\sigma_1^{t}\mcal{L}_{w}^{t}f\|_{L^{2}(\mcal{X}_{t\epsilon},g^{p+r})}\le C_1 M\epsilon.
    \end{align*}
    \item[(2)] If $s$ is even, then we have with $t=(s-2)/2$:
    \begin{align*}
        \|L_{w,\epsilon}^{t}f-\sigma_1^{t}\mcal{L}_{w}^{t}f\|_{L^{2}(\mcal{X}_{t\epsilon},g^{p+r})}\le C_2 M\epsilon^{2}.
    \end{align*}
\end{itemize}
\end{Lemma}

\begin{proof}
Without loss of generality, we assume both $g$ and $f$ are $C^{\infty}(\mcal{X})$ due to the fact that $C^{\infty}(\mcal{X})$ is dense in both $H^{s}(\mcal{X},g)$ and $C^{s-1}(\mcal{X})$ and the norm in the statements is continuous with respect to $\|\cdot\|_{H^{s}(\mcal{X},g)}$ and $\|\cdot\|_{C^{s-1}(\mcal{X})}$. 

Actually, we claim the following stronger result: for $t<s/2$ and every $x\in \mcal{X}_{t\epsilon}$,
\begin{align}\label{inter}
    L_{w,\epsilon}^{t}f(x)=\sigma_1^{t}\mcal{L}_{w}f(x)+
    \sum_{j=1}^{\lfloor (s-1)/2\rfloor-t}r_{2(j+t)}(x)\epsilon^{2j}+r_{s}(x)\epsilon^{s-2t},
\end{align}
for some functions $r_{j}$ such that
\begin{align}\label{rjbound}
    \|r_j\|_{H^{s-j}(\mcal{X}_{t\epsilon},g)}\le C\|g\|_{C^{s-1}(\mcal{X})}^{t}M.
\end{align}
Note that the dependence of the functions $r_{j}$ on $t$ is suppressed in the notation.

The key idea underlying the proof of (\ref{inter}) is to consider the following Taylor expansion. For an $s$-times differentiable function $F:\mcal{X}\rightarrow \mbb{R}$ and $x\in\mcal{X}$, define the following operator $d_{x}^{s}$:
\begin{align*}
    (d_{x}^{s}F)(z):=\sum_{|\alpha|=s}D^{\alpha}F(x)z^{\alpha}.
\end{align*}
Also, define $d^{s}F:=\sum_{|\alpha|=s}D^{\alpha}F$. Then, for $\phi\in C^{s}(\mcal{X})$ and some $h>0$, $z\in\mcal{X}_{h}$, $x\in B_{z}(h)$, the Taylor expansion at $z$ is given as:
\begin{align*}
    \phi(x)=\phi(z)+\sum_{j=1}^{s-1}\frac{1}{j!}(d_{x}^{j}\phi)(x-z)+R_{s}(x,z;\phi).
\end{align*}
Here, we note that $(d_{x}^{j}\phi)(z)$ is a polynomial of degree $j$ and we have for any $y\in\mbb{R}$:
\begin{align*}
    (d_{x}^{j}\phi)(yz)=y^j (d_{x}^{j}\phi)(z).
\end{align*}
The remainder term $R_{j}(x,z;\phi)$ is 
\begin{align*}
    R_{j}(x,z;\phi):=\frac{1}{(j-1)!}\int_{0}^{1}(1-\theta)^{j-1}(d_{z+\theta(x-z)}^{j}\phi)(x-z)d\theta,
\end{align*}
such that for any $x^{*}\in B_{\mbf{0}}(1)$,
\begin{align*}
    \underset{x\in\mcal{X}_{h}}{\sup}|R_{j}(x,x+hx^{*};\phi)|\le Ch^{j}\|\phi\|_{C^{j}(\mcal{X})},
\end{align*}
and
\begin{align*}
    \int_{\mcal{X}_{h}}|R_{j}(z+\theta hx,z;\phi)|^{2}dz\le h^{2j}\int_{\mcal{X}_{h}}\int_{0}^{1}|(d_{z+\theta hx}^{j}\phi)(z)|^{2}d\theta dz\le h^{2j}\|d^{j}\phi\|_{L^{2}(\mcal{X})}^{2}.
\end{align*}

Now, we apply the above Taylor expansion on the function $f_{g}(x):=f(x)/g(x)^{r}$ up to order $s$ and the function $g^{q/2}(x)$ up to order $S$ in $L_{w,\epsilon}f(x)$, where $S=1$ if $s=1$ and otherwise $S=s-1$ and obtain:
\begin{align*}
    L_{w,\epsilon}f(x)&=\frac{1}{\epsilon^{d+2}}\sum_{j_1=1}^{s-1}\sum_{j_2=0}^{S-1}\frac{1}{j_1!j_2!}\int_{\mcal{X}}g(x)^{q/2-p}\eta\left(\frac{\|x-z\|}{\epsilon}\right)(d_{x}^{j_1}f_{g})(x-z)(d_{x}^{j_2}g^{q/2})(z-x)dz\\&+\frac{1}{\epsilon^{d+2}}\sum_{j=1}^{s-1}\frac{1}{j!}\int_{\mcal{X}}g(x)^{q/2-p}\eta\left(\frac{\|x-z\|}{\epsilon}\right)(d_{x}^{j_1}f_{g})(x-z)R_{S}(x,z;g^{q/2})dz\\&+\frac{1}{\epsilon^{d+2}}\int_{\mcal{X}}g(x)^{p/2-p}\eta\left(\frac{\|x-z\|}{\epsilon}\right)R_{s}(x,z;f_g)g(z)^{q/2}dz.
\end{align*}
Now, with the transformation $y=(z-x)/\epsilon$, we have
\begin{align*}
    L_{w,\epsilon}f(x)&=-\frac{1}{\epsilon^{2}}\sum_{j_1=1}^{s-1}\sum_{j_2=0}^{S-1}\frac{\epsilon^{j_1+j_2}}{j_1!j_2!}\int_{B_{\mbf{0}}(1)}g(x)^{q/2-p}\eta\left(\|y\|\right)(d_{x}^{j_1}f_{g})(y)(d_{x}^{j_2}g^{q/2})(y)dy\\&-\frac{1}{\epsilon^{2}}\sum_{j=1}^{s-1}\frac{\epsilon^{j}}{j!}\int_{B_{\mbf{0}}(1)}g(x)^{q/2-p}\eta\left(\|y\|\right)(d_{x}^{j_1}f_{g})(y)R_{S}(x,\epsilon y+x;g^{q/2})dy\\&+\frac{1}{\epsilon^{2}}\int_{B_{\mbf{0}}(1)}g(x)^{p/2-p}\eta\left(\|y\|\right)R_{s}(x,\epsilon y+x;f_g)g(\epsilon y+x)^{q/2}dy\\&=:L_{1}(x)+L_{2}(x)+L_{3}(x).
\end{align*}
We will now prove \eqref{inter} by induction on $t,$ and throughout this proof, with a slight abuse of notation, the functions $r_{j}$ in \eqref{inter} may vary from line to line depending on $t$ at the induction step but they will always satisfy the condition \eqref{rjbound} as we are only interested in the bounds.

Firstly, we start with $L_{1}(x)$. If $s=1$, we can see $L_{1}(x)=0$. Therefore, in the following, we only focus on $s\ge 2$. Now, we define 
\begin{align*}
    l_{j_1,j_2}(x):=\int_{B_{\mbf{0}}(1)}g(x)^{q/2-p}\eta(\|y\|)(d_{x}^{j_1}f_{g})(y)(d_{x}^{j_2}g^{q/2})(y)dy,
\end{align*}
such that 
\begin{align*}
    L_{1}(x)=-\frac{1}{\epsilon^{2}}\sum_{j_1=1}^{s-1}\sum_{j_2=0}^{S-1}\frac{\epsilon^{j_1+j_2}}{j_1!j_2!}l_{j_1,j_2}(x).
\end{align*}
Since $(d_{x}^{j}f_{g})(y)$ is a polynomial of degree $j$, $l_{j_1,j_2}$ actually depends on the sum $j_1+j_2$ and $d_{x}^{j_1}d_{x}^{j_2}$ is an order $j_1+j_2$ multivariate monomial. Therefore, when $j_1+j_2$ is odd, we have
\begin{align*}
    l_{j_1,j_2}(x)=0.
\end{align*}
Then, when $s=2$, we have $j_1+j_2=1$ and $L_{1}(x)=0$. As for $s\ge 3$, we notice that the lowest order term of $L_{11}(x)$ is from $j_1+j_2=2$, which means either $j_1=1,j_2=1$ or $j_1=2,j_2=0$. We have
\begin{align*}
    l_{1,1}(x)&=\int_{B_{\mbf{0}}(1)}g(x)^{q/2-p}\eta(\|y\|)(d_{x}^{1}f_{g})(y)(d_{x}^{1}g^{q/2})(y)dy\\&=\sum_{i_1=1,i_2=1}^{d}g(x)^{q/2-p}(Df_{g}(x))^{(i_1)}(Dg^{q/2}(x))^{(i_{2})}\int_{B_{\mbf{0}}(1)}\|y\|^{2}\eta(\|y\|)dy,
\end{align*}
and
\begin{align*}
    \frac{1}{2}l_{2,0}(x)&=\frac{1}{2}\int_{B_{\mbf{0}}(1)}g(x)^{q/2-p}\eta(\|y\|)(d_{x}^{2}f_{g})(y)(d_{x}^{2}g^{q/2})(y)dy\\&\\&=\frac{1}{2}\sum_{i=1}^{d}g(x)^{q/2-p}((Df_{g}(x))^{(i)})^{2}g(x)^{q/2}\int_{B_{\mbf{0}}(1)}\|y\|^{2}\eta(\|y\|)dy.
\end{align*}
Therefore, we have by definition:
\begin{align*}
    \mcal{L}_{w}f(x)=-\frac{1}{2g(x)^{p}}\left(\nabla g(x)^{q}\cdot \nabla\left(\frac{f(x)}{g(x)^{r}}\right)+g(x)^{q}\Delta\left(\frac{f(x)}{g(x)^{r}}\right)\right),
\end{align*}
and
\begin{align*}
    -(l_{1,1}(x)+\frac{1}{2}l_{2,0}(x))=\sigma_1\mcal{L}_{w}f(x).
\end{align*}
This is exactly the leading term. \textcolor{black}{We remark here that in \cite[Section D.2]{green2021minimax}, the negative sign is missing, which does not actually give the Laplacian operator by the leading term.} Now, it remains to bound the higher order terms with $j_1+j_2>2$. We will show that
\begin{align*}
    L_{1}(x)=\sigma_1\mcal{L}_{w}+\sum_{j=1}^{\lfloor (s-1)/2\rfloor-1}r_{2(j+1)}(x)\epsilon^{2j}+r_{s}(x)\epsilon^{s-2}.
\end{align*}
It suffices to show for $j_1+j_2>2$, $l_{j_1,j_2}$ satisfies \eqref{rjbound} for $j=\min\{j_1+j_2-2,s-2\}$. Through the multi-index notation, we write that
\begin{align*}
    l_{j_1,j_2}(x)=g(x)^{q/2-p}\sum_{|\alpha_1|=j_1,|\alpha_2|=j_2}D^{\alpha_1}f_g(x)D^{\alpha_2}g^{q/2}(x)\int_{B_{\mbf{0}}(1)}y^{\alpha_1}y^{\alpha_2}\eta(\|y\|)dy,
\end{align*}
where $|\int_{B_{\mbf{0}}(1)}y^{\alpha_1}y^{\alpha_2}\eta(\|y\|)dy|<\infty$ for all $\alpha_1,\alpha_2$. Then, by H\"{o}lder's inequality, we have for $|\alpha_1|=j_1$, $|\alpha_{2}|=j_2$,
\begin{align*}
    \|g(x)^{q/2-p}D^{\alpha_1}f_g D^{\alpha_2}g^{q/2}\|_{H^{s-(j+2)}(\mcal{X},g)}&\lesssim \|D^{\alpha_1}f_g \|_{H^{s-(j+2)}(\mcal{X},g)}\|g(x)^{q/2-p}D^{\alpha_2}g^{q/2}\|_{C^{s-(j+2)}(\mcal{X})}\\&\lesssim \|D^{\alpha_1}f_g \|_{H^{s-j_1}(\mcal{X},g)}\|D^{\alpha_2}g^{q/2}\|_{C^{s-(j_2+1)}(\mcal{X})}\\&\le M\|g\|_{C^{s-1}}.
\end{align*}
Summing over all $|\alpha_1|=j_1$ and $|\alpha_2|=j_2$, we obtain that $l_{j_1,j_2}$ satisfies \eqref{rjbound}.

Next, as for $L_2(x)$, note that if $s=1$, $L_{2}(x)=0$. We want to show that for $s\ge 2$, 
\begin{align*}
    \|L_{2}\|_{L^{2}(\mcal{X}_{\epsilon},g^{p+r})}\le C\epsilon^{s-2}M\|g\|_{C^{s-1}(\mcal{X})}.
\end{align*}
Clearly, if $s=1$, $L_2(x)=0$. Now, for $s\ge 2$, we have $S=s-1$ and since $|R_{s-1}(x,x+\epsilon x^{*})|\le C\epsilon^{s-1}\|g\|_{C^{s-1}(\mcal{X})}$ for any $x^{*}\in B_{\mbf{0}}(1)$ and $d_{x}^{j}(\cdot)$ is a $j$-homogeneous function, we have
\begin{align*}
    |L_{2}(x)|&\le \sum_{j=1}^{s-1}\frac{\epsilon^{j-2}}{j!}\int_{B_{\mbf{0}}(1)}g(x)^{q/2-p}\eta\left(\|y\|\right)|(d_{x}^{j_1}f_{g})(y)|\cdot|R_{S}(x,\epsilon y+x;g^{q/2})|dy\\&\le C\epsilon^{s-2}\|g\|_{C^{s-1}(\mcal{X})}\sum_{j=1}^{s-1}\frac{1}{j!}\int_{B_{\mbf{0}}(1)}g(x)^{q/2-p}\eta\left(\|y\|\right)|(d_{x}^{j_1}f_{g})(y)|dy.
\end{align*}
Moreover, we have by Cauchy–Schwarz inequality,
\begin{align*}
    &\int_{\mcal{X}_{\epsilon}}g(x)^{p+r}\left(\int_{B_{\mbf{0}}(1)}g(x)^{q/2-p}\eta\left(\|y\|\right)|(d_{x}^{j_1}f_{g})(y)|dy\right)^{2}dx\\&\le \int_{\mcal{X}_{\epsilon}}g(x)^{q-p+r}\left(\int_{B_{\mbf{0}}(1)}\eta\left(\|y\|\right)|(d_{x}^{j_1}f_{g})(y)|^2dy\right)\left(\int_{B_{\mbf{0}}(1)}\eta(\|y\|)dy\right)dx\\&\le \sigma_0 \int_{B_{\mbf{0}}(1)}\int_{\mcal{X}_{\epsilon}}g(x)^{q-p+r}\eta(\|y\|)((d^{j}f_g)(x))^{2}dxdy\\&\lesssim \sigma_0^{2}\int_{\mcal{X}_{\epsilon}}g(x)^{p+r}((d^{j}f_g)(x))^{2}dx\\&=\sigma_{0}^{2}\|d^{j}f_g\|_{L^{2}(\mcal{X}_{\epsilon},g^{p+r})}^{2},
\end{align*}
where in the last step, we use the fact that $|d_{x}^{j}f(y)|\le |d^{j}f(x)|$ for all $y\in B_{\mbf{0}}(1)$. Therefore, it yields that
\begin{align*}
    &\int_{\mcal{X}_{\epsilon}}g(x)^{p+r}|L_{2}(x)|^{2}dx\\&\le C\left(\epsilon^{s-2}\|g\|_{C^{s-1}(\mcal{X})}\right)^{2}\sum_{j=1}^{s-1}\int_{\mcal{X}_{\epsilon}}g(x)^{p+r}\left(\frac{1}{j!}\int_{B_{\mbf{0}}(1)}g(x)^{q/2-p}\eta\left(\|y\|\right)|(d_{x}^{j_1}f_{g})(y)|dy\right)^{2}dx\\&\le C\left(\epsilon^{s-2}\|g\|_{C^{s-1}(\mcal{X})}\right)^{2}\sum_{j=1}^{s-1}\|d^{j}f_{g}\|_{L^{2}(\mcal{X}_{\epsilon},g^{p+r})}^{2}.
\end{align*}
We obtain the desired bound. 

Finally, similar to $L_{2}(x)$, we obtain the same bound for $L_3(x).$ Combining the obtained bounds for $L_1(x) - L_3(x),$ we obtain \eqref{inter} for $t=1$. 

Now, we perform the induction step. Assuming the bound \eqref{inter} holds up to some $t<s/2$, we want to show it also holds for $t+1$, with $t+1 < s/2$. For convenience, we introduce the following notation: for any $1\le j\le l\le s$, denote by $r_{j,l}(x)=r_{(s-l)+j}(x)$. Note again that the functions $r_{j,l}$ implicitly depend on $t$ at the induction step thus they may vary in the below arguments from line to line. For a function $r\in H^{l}(\mcal{X}_{t\epsilon},g;C\|g\|_{C^{s-1}(\mcal{X})}^{t}M)$ for some $l\le s$, if $l\le 2$, we have by the inductive hypothesis that for any $x\in\mcal{X}_{(t+1)\epsilon}$, 
\begin{align*}
    L_{w,\epsilon}r(x)=r_{l}^{l}(x)\epsilon^{l-2}.
\end{align*}
On the other hand, if $2<l\le s$, then by the inductive hypothesis, it holds that for any $x\in\mcal{X}_{(t+1)\epsilon}$,
\begin{align}\label{inductive1}
    L_{w,\epsilon}r(x)=\sigma_{1}\mcal{L}_{w}r(x)+\sum_{j=1}^{\lfloor (l-1)/2\rfloor-1}r_{2j+2,l}(x)\epsilon^{2j}+r_{l,l}(x)\epsilon^{l-2}.
\end{align}
Then, we have
\begin{align}\label{inductive2}
    L_{w,\epsilon}^{t+1}f(x)&=(L_{w,\epsilon}\circ L_{w,\epsilon}^{t}f)(x)\nonumber\\&=\sigma_1^{t}L_{w,\epsilon}\mcal{L}_{w}^{t}f(x)+\sum_{j=1}^{\lfloor (s-1)/2\rfloor-k}L_{w,\epsilon}r_{2(j+t)}(x)\epsilon^{2j}+L_{w,\epsilon}r_{s}(x)\epsilon^{s-2t}.
\end{align}
In the following, we will bound these terms on the right-hand side individually. First of all, since $\mcal{L}_{w}^{t}f\in H^{s-2t}(\mcal{X},g;C\|g\|_{C^{s-1}(\mcal{X})}^{t}M)$, applying \eqref{inductive1} yields 
\begin{align}\label{t+11}
    L_{w,\epsilon}\mcal{L}_{w}^{t}f(x)&=\sigma_1 \mcal{L}_{w}^{t+1}f(x)+\sum_{j=1}^{\lfloor (s-2t-1)/2\rfloor-1}r_{2j+2,s-2t}(x)\epsilon^{2j}+r_{s-2t,s-2t}(x)\epsilon^{s-2t-2}\nonumber\\&=\sigma_1 \mcal{L}_{w}^{t+1}f(x)+\sum_{j=1}^{\lfloor (s-1)/2\rfloor-(t+1)}r_{2(t+1+j)}(x)\epsilon^{2j}+r_{s}(x)\epsilon^{s-2(t+1)},
\end{align}
where we apply the fact mentioned before that $r_{j,l}(x)=r_{(s-l)+j}(x)$.

Next, suppose $j<\lfloor (s-1)/2\rfloor-t$. We apply \eqref{inductive1} and obtain
\begin{align*}
    L_{w,\epsilon}r_{2(j+t)}(x)&=\sigma_1 \mcal{L}_{w}r_{2(j+t)}(x)+\sum_{i=1}^{\lfloor (s-2j-2t-1)/2\rfloor-1}r_{2i+2,s-2(j+t)}(x)\epsilon^{2i}\\&\quad+r_{s-2(j+t),s-2(j+t)}(x)\epsilon^{s-2(j+t)-2}\\&=r_{2(j+t+1)}(x)+\sum_{i=1}^{\lfloor (s-1)/2\rfloor-(j+t+1)}r_{2(i+j+t+1)}(x)\epsilon^{2i}+r_{s}(x)\epsilon^{s-2(j+t+1)},
\end{align*}
where we use the fact that $r_{j,l}(x)=r_{(s-l)+j}(x)$ and $\sigma_1\mcal{L}_w r_{2(j+t)}(x)=r_{2,s-2(j+t)}(x)=r_{2(j+t+1)}(x)$. Therefore, we have
\begin{align}\label{t+12}
    L_{w,\epsilon}r_{2(j+t)}(x)\epsilon^{2j}=r_{2(j+t+1)}(x)\epsilon^{2j}+\sum_{m=1}^{\lfloor (s-1)/2\rfloor-(k+1)}r_{2(m+t+1)}(x)\epsilon^{2m}+r_{s}(x)\epsilon^{s-2(k+1)},
\end{align}
where the last equality is by changing the variable $m=i+j$. Moreover, when $j=\lfloor (s-1)/2\rfloor-t$, we have $2(j+t)=2\lfloor (s-1)/2\rfloor$ and we simply calculate that
\begin{align}\label{t+13}
    L_{w,\epsilon}r_{2(j+t)}(x)\epsilon^{2j}=r_{s-2(j+t)}^{s-2(j+t)}(x)\epsilon^{s-2(j+k)}\epsilon^{2j}=r_{s}(x)\epsilon^{s-2(k+1)}.
\end{align}
Finally, according to \eqref{inductive1}, we have
\begin{align}\label{t+14}
    L_{w,\epsilon}r_{s}(x)\epsilon^{s-2t}=r_{s}(x)\epsilon^{s-2(t+1)}.
\end{align}
Combining \eqref{t+11}-\eqref{t+14} with \eqref{inductive2}, we obtain the proof for $t+1$.

\end{proof}

Recall that we write $\mcal X = \mcal{X}_{t\epsilon} \sqcup \partial  \mcal{X}_{t\epsilon}$, where for any $x\in \mcal{X}_{t\epsilon}$, $B_{x}(t\epsilon)\subset \mcal{X}$ and $\partial_{t\epsilon} \mcal{X}$ as its complement within $\mcal{X}$ consisting of points `close' to the boundary.

\begin{Lemma}\label{lemma62boundary}
    For $f\in H^{s}_{0}(\mcal{X},g;M)$ and $t>0$ such that $2t<s$, there exists a constant $c>0$ not depending on $M$ or $f$ such that for all $\epsilon<c$,
\begin{align*}
    \|L_{w,\epsilon}^{t}f\|_{L^{2}(\partial_{t\epsilon} \mcal{X},g^{p+r})}^{2}\lesssim \epsilon^{2(s-2t)}M^{2}.
\end{align*}
\end{Lemma}

\begin{proof}
    Note that according to Lemma \ref{lemma6}, we have
\begin{align*}
    \|L_{w,\epsilon}^{t}f\|_{L^{2}(\partial_{t\epsilon} \mcal{X},g^{p+r})}^{2}\lesssim \frac{1}{\epsilon^{4}} \|L_{w,\epsilon}^{t-1}f\|_{L^{2}(\partial_{t\epsilon} \mcal{X},g^{p+r})}^{2}\lesssim\ldots\lesssim \frac{1}{\epsilon^{4t}}\|f_{g}\|_{L^{2}(\partial_{t\epsilon} \mcal{X},g^{p+r})}^{2}.
\end{align*}
Therefore, it suffices to show for all $\epsilon<c$,
\begin{align}\label{boundarybound}
    \|f_{g}\|_{L^{2}(\partial_{t\epsilon} \mcal{X},g^{p+r})}^{2}\lesssim\epsilon^{2s}\|f\|_{H^{s}(\mcal{X},g)}^{2}.
\end{align}
In order to deal with $f_{g}$ near the boundary, we will take a similar procedure used in \cite[Proof of Lemma 5]{green2021minimax} and \cite[Theorem 18.1]{leoni2017first} as follows. With loss of generality, we take $t=1$ as one can view $\epsilon<c/t$ for proving for the general case.

\emph{Step I: Local patch.} We assume that for some $c_0>0$ and a Lipschitz mapping $\phi:\mbb{R}^{d-1}\rightarrow [-c_0,c_0]$ and since $f\in H^{s}_{0}(\mcal{X},g;M)$, without loss of generality, we can assume that $f_{g}\in C_{c}^{\infty}(U_{\psi}(c_0))$ with 
\begin{align*}
    U_{\psi}(c_0):=\{y\in Q(0,c_0):\psi(y^{(-d)})\le y^{(d)}\},
\end{align*}
where $Q(0,c_0)$ is the $d$-dimensional hypercube of side length $c_0$ centered at $\mbf{0}$. Now, following step 1 in \cite[Proof of Lemma 5]{green2021minimax} by replacing $f$ as $f_{g}$, we have
\begin{align*}
    |f_{g}(y)|^2&\lesssim \epsilon^{2(s-1)}\left(\int_{\psi(y^{(-d)})}^{y^{(d)}}|(D^sf_{g}(y^{(-d)},z))^{(d)}|dz\right)^{2}\\&\lesssim \epsilon^{2s-1}\int_{\psi(y^{(-d)})}^{y^{(d)}}|(D^sf_{g}(y^{(-d)},z))^{(d)}|^{2}dz.
\end{align*}
Then, we obtain:
\begin{align}\label{step1:1}
    \int_{V_{\psi}(\epsilon)}g(y)^{p+r}|f_{g}(y)|^2dy&\lesssim \int_{Q_{d-1}(c_0)}\int_{\psi(y^{(-d)})}^{\psi(y^{(-d)})+\epsilon}|f_{g}(y^{(-d)},y^{(d)})|^{2}dy^{(d)}dy^{(-d)}\nonumber\\&\lesssim \epsilon^{2s-1}\int_{Q_{d-1}(c_0)}\int_{\psi(y^{(-d)})}^{\psi(y^{(-d)})+\epsilon}\int_{\psi(y^{(-d)})}^{y^{(d)}}|(D^sf_{g}(y^{(-d)},z))^{(d)}|^{2}dzdy^{(d)}dy^{(-d)},
\end{align}
where $Q_{d-1}(0,c_0)$ is the $d$-1 dimensional hypercube of side length $c_0$ centered at $\mbf{0}$. Also, by changing the integration order, it yields that
\begin{align}\label{step1:2}
    \int_{\psi(y^{(-d)})}^{\psi(y^{(-d)})+\epsilon}\int_{\psi(y^{(-d)})}^{y^{(d)}}|(D^sf_{g}(y^{(-d)},z))^{(d)}|^{2}dzdy^{(d)}&\lesssim \epsilon\int_{\psi(y^{(-d)})}^{\psi(y^{(-d)})+\epsilon}|(D^sf_{g}(y^{(-d)},z))^{(d)}|^{2}dz\nonumber\\&\lesssim \epsilon\int_{\psi(y^{(-d)})}^{c_0}|(D^sf_{g}(y^{(-d)},z))^{(d)}|^{2}dz.
\end{align}
Combining \eqref{step1:1} and \eqref{step1:2}, we obtain:
\begin{align*}
    \int_{V_{\psi}(\epsilon)}g(y)^{p+r}|f_{g}(y)|^2dy&\lesssim \epsilon^{2s}\int_{Q_{d-1}(c_0)}\int_{\psi(y^{(-d)})}^{c_0}g(y^{(-d)},z)^{q}|(D^sf_{g}(y^{(-d)},z))^{(d)}|^{2}dzdy^{(-d)}\\&\lesssim \epsilon^{2s}\|f\|_{H^{s}(U_{\psi}(c_0),g)}^{2}.
\end{align*}

\emph{Step 2: Rigid motion of local patch} Now suppose at a point $x_0\in \partial \mcal{X}$, there exits a rigid motion $T:\mbb{R}^{d}\rightarrow \mbb{R}^{d}$ such that $T(x_0)=0$, and a number $C_0$ such that we have all $C_0\epsilon\le c_0 $,
\begin{align}\label{localrigid}
    T(Q_{T}(x_0,c_0)\cap \partial_{\epsilon} \mcal{X})\subseteq V_{\psi}(C_0\epsilon)\quad \text{and}\quad T(Q_{T}(x_0,c_0)\cap \mcal{X})=U_{\psi}(c_0),
\end{align}
where $Q_{T}(x_0,c_0)$ is a hypercube in $\mbb{R}^{d}$ of side length $c_0$ centered at $x_0$ (not necessarily coordinate-axis-aligned). Let $v_{g}(y):=f_{g}(T^{-1}(y))$ and $v(y):=f(T^{-1}(y))$ for all $y\in U_{\psi}(c_0)$. Then, if $f_{g}\in C_{c}^{\infty}(\mcal{X})$, we have $v_{g}\in C_{c}^{\infty}(U_{\psi}(c_0))$ such that $\|v_{g}\|_{H^{s}(U_{\psi}(c_0))}^{2}=\|f_g\|_{H^{s}(Q_{T}(x_0,c_0))\cap \mcal{X}}^{2}$. Therefore, according to Step 1, we have
\begin{align*}
    \int_{V_{\psi}(C_0\epsilon)}g(x)^{p+r}|v_g(y)|^2dy\lesssim \epsilon^{2s}\|v\|_{H^{s}(U_{\psi}(c_0)),g}^{2}.
\end{align*}
Then, it yields that
\begin{align*}
    &\int_{Q_{T}(x_0,c_0)\cap \partial_{\epsilon} \mcal{X}}g^{p+r}(x)|f_g(x)|^{2}dx\\&=\int_{T(Q_{T}(x_0,c_0)\cap \partial_{\epsilon} \mcal{X})}g^{p+r}(y)|v_g(y)|^{2}dx\\&\lesssim \int_{V_{\psi}(C_0\epsilon)}g(x)^{p+r}|v_g(y)|^2dy\\&\lesssim \epsilon^{2s}\|v\|_{H^{s}(U_{\psi}(c_0),g)}^{2}\\&\lesssim \epsilon^{2s}\|f\|_{H^{s}(Q_{T}(x_0,c_0)\cap \partial_{\epsilon} \mcal{X},g)}^{2}\lesssim \epsilon^{2s}\|f\|_{H^{s}(\mcal{X},g)}^{2}.
\end{align*}

\emph{Step 3: Lipschitz domain.} Now we arrive at the last step where we shall deal with the case: $\mcal{X}$ is assumed to be an open, bounded subset of $\mbb{R}^{d}$ with Lipschitz boundary. Again, following the procedure in \cite[Proof of Lemma 5]{green2021minimax}. In this case, for every $x_0\in \partial\mcal{X}$, there exists a rigid motion $T_{x_0}:\mbb{R}^{d}\rightarrow\mbb{R}^{d}$ such that $T_{x_0}(x_0)=0$, a number $c_0(x_0)$, a Lipshitz mapping $\psi_{x_0}:\mbb{R}^{d-1}\rightarrow [-c_0(x_0),c_0(x_0)]$ and a number $C_{0}(x_0)$ satisfying for all $C_0(x_0)\epsilon\le c_0(x_0)$, \eqref{localrigid} holds for replacing $c_0,C_0,T,\psi$ by $c_0(x_0),C_0(x_0),T_{x_0},\psi_{x_0}$ respectively. Therefore, by Step 2, we have
\begin{align*}
    \int_{Q_{T_{x_0}}(x_0,c_0(x_0))\cap \partial_{\epsilon} \mcal{X}}g^{p+r}(x)|f_g(x)|^{2}dx\lesssim_{x_0} \epsilon^{2s}\|f\|_{H^{s}(\mcal{X},g)}^{2}.
\end{align*}
Although the constant in the last bound depends on $x_0$, by compactness assumption, there exists a finite subset (denoted by $x_{0,1},\ldots,x_{0,N}$) of the collection of hypercubes $\{Q_{T_{x_0}}(x_0,c_0(x_0)/2):x_{0}\in\partial\mcal{X}\}$ which covers $\partial\mcal{X}$. Then, by taking the minimum of all constants with respect to $x_{0,1},\ldots,x_{0,N}$, we can conclude that
\begin{align*}
    \partial_{\epsilon}(\mcal{X})\subseteq \bigcup_{i=1}^{N}Q_{T_{x_{0,i}}}(x_{0,i},c_{0}(x_{0,i})).
\end{align*}
Consequently, we have
\begin{align*}
    \int_{\partial \mcal{X}}g^{p+r}(x)|f_{g}(x)|^{2}dx\lesssim \sum_{i=1}^{N}\int_{Q_{T_{x_{0,i}}}(x_{0,i},c_0(x_{0,i}))\cap \partial_{\epsilon} \mcal{X}}g^{p+r}(x)|f_g(x)|^{2}dx\lesssim \epsilon^{2s}\|f\|_{H^{s}(\mcal{X},g)}^{2}.
\end{align*}
Therefore, we proved the desired result \eqref{boundarybound}.
\end{proof}

The following result presents a higher order version of Lemma \ref{lemma61} for $s>1$ and the non-local weighted Sobolev seminorm, $\langle L_{w,\epsilon}^{s}f,f \rangle_{g^{p-r}}$.
\begin{Lemma}\label{lemma62}
    For $f\in H^{1}(\mcal{X},g;M)\cap H_{0}^{s}(\mcal{X},g;M)$ with $s\in\mbb{N}_{+}\backslash\{1\}$, we have
    \begin{align*}
        \langle L_{w,\epsilon}^{s}f,f\rangle_{g^{p-r}}\lesssim M^{2}.
    \end{align*}
\end{Lemma}

\begin{proof}[Proof of Lemma \ref{lemma62}]
    \textcolor{black}{Note here that we fix the assumption that $f\in H^{1}(\mcal{X},g;M)$ besides $f\in H_{0}^{s}(\mcal{X},g;M)$, which is missing in the statement of \cite[Theorem 3]{green2021minimax}. In general, it is not true for $\mcal{X}\neq\mbb{R}^{d}$ that $H_{0}^{1}(\mcal{X},g;M)=H^{1}(\mcal{X},g;M)$.} Based on Lemma \ref{lemma61}, we will prove Lemma \ref{lemma62} in a recursive way for $s>1$. Recall that $L_{w,\epsilon}$ is self-adjoint with respect to the weighted inner product, meaning $\langle L_{w,\epsilon}f_1,f_2\rangle_{g^{p-r}}=\langle f_1,L_{w,\epsilon}f_2\rangle_{g^{p-r}}$ for any $f_1/g^r,f_2/g^r\in L^{2}(\mcal{X},g^{p+r})$. Also recall the definition of the Dirichlet energy given in Lemma~\ref{lemma64}, which can be stated as $E_{w,\epsilon}(f,\mcal{X}) = 2\langle L_{w,\epsilon}f,f\rangle_{g^{p-r}}$.

Following the procedure in \cite{green2021minimax} \footnote{We remark here that the factor 2 is missing in \cite[Section D.4]{green2021minimax}.}, when $s=2t+1$ for $t\ge 1$, by using self-adjointness, we have
\begin{align*}
    \langle L_{w,\epsilon}^{s}f,f\rangle_{g^{p-r}}=\langle L_{w,\epsilon}^{t+1}f,L_{w,\epsilon}^{t}f\rangle_{g^{p-r}}=\frac{1}{2}E_{w,\epsilon}(L_{w,\epsilon}^{t}f,\mcal{X}).
\end{align*}
We divide the Dirichlet energy into two parts:
\begin{align*}
    &E_{w,\epsilon}(L_{w,\epsilon}^{t}f,\mcal{X})\\&=\frac{1}{\epsilon^{d+2}}\int_{\mcal{X}_{t\epsilon}}\int_{\mcal{X}_{t\epsilon}}(g(x)^{-r}f(x)-g(z)^{-r}f(z))^2\frac{\eta\left(\frac{\|x-z\|}{\epsilon}\right)}{g(x)^{1-q/2}g(z)^{1-q/2}}g(x)g(z)dxdz\\&\quad+\frac{1}{\epsilon^{d+2}}\int_{\partial_{t\epsilon} \mcal{X}}\int_{\partial_{t\epsilon} \mcal{X}}(g(x)^{-r}f(x)-g(z)^{-r}f(z))^2\frac{\eta\left(\frac{\|x-z\|}{\epsilon}\right)}{g(x)^{1-q/2}g(z)^{1-q/2}}g(x)g(z)dxdz\\&=:E_{w,\epsilon}(L_{w,\epsilon}^{t}f,\mcal{X}_{t\epsilon})+E_{w,\epsilon}(L_{w,\epsilon}^{t}f,\partial_{t\epsilon} \mcal{X}),
\end{align*}
where $\mcal{X}_{t\epsilon}$ and $\partial \mcal{X}_{t\epsilon}$ have been introduced right before Lemma~\ref{lemma62interior} ($\partial \mcal{X}_{t\epsilon} \subset \mcal X$ consists of points $t\epsilon$-close to the boundary of $\mcal X$, and $\mcal{X}_{t\epsilon} = \mcal X \setminus \partial \mcal{X}_{t\epsilon}$). 

By Jensen's inequality, we have
\begin{align*}
    E_{w,\epsilon}(L_{w,\epsilon}^{t}f,\mcal{X}_{t\epsilon})&\le 3\sigma_{1}^{2t}E_{w,\epsilon}\left(\sigma_{1}^{t}L_{w}^{t}f,\mcal{X}_{t\epsilon}\right)\\&\quad+\frac{6}{\epsilon^{d+2}}\int_{\mcal{X}_{t\epsilon}}\int_{\mcal{X}_{t\epsilon}}\left(g(x)^{-r}L_{w,\epsilon}^{t}f(x)-g(z)^{-r}\sigma_{1}^{t}L_{w}^{t}f(z)\right)^2\\&\qquad\qquad\qquad\frac{\eta\left(\frac{\|x-z\|}{\epsilon}\right)}{g(x)^{1-q/2}g(z)^{1-q/2}}g(x)g(z)dxdz.
\end{align*}
By definition \eqref{weightedop}, we have $\mcal{L}_{w}^{t}f\in H^{1}(\mcal{X},g;C\|g\|_{C^{s-1}(\mcal{X})}^{t}M)$ for some constant $C>0$, an application of Lemma \ref{lemma61} shows $E_{w,\epsilon}\left(\sigma_{1}^{t}L_{w}^{t}f,\mcal{X}_{t\epsilon}\right)\lesssim M^{2}$. We then focus on the second term on the right-hand side of the above inequality. According to Lemma \ref{lemma62interior}, we obtain:
\begin{align*}
    &\frac{1}{\epsilon^{d+2}}\int_{\mcal{X}_{t\epsilon}}\int_{\mcal{X}_{t\epsilon}}\left(g(x)^{-r}L_{w,\epsilon}^{t}f(x)-g(z)^{-r}\sigma_{1}^{t}\mcal{L}_{w}^{t}f(z)\right)^2\frac{\eta\left(\frac{\|x-z\|}{\epsilon}\right)}{g(x)^{1-q/2}g(z)^{1-q/2}}g(x)g(z)dxdz\\&\lesssim \frac{1}{\epsilon^{d+2}}\int_{\mcal{X}_{t\epsilon}}\int_{\mcal{X}_{t\epsilon}}g(x)^{p+r}\left(g(x)^{-r}L_{w,\epsilon}^{t}f(x)-g(x)^{-r}\sigma_{1}^{t}\mcal{L}_{w}^{t}f(x)\right)^2\eta\left(\frac{\|x-z\|}{\epsilon}\right)dxdz\\&\lesssim \frac{1}{\epsilon^{2}}\int_{\mcal{X}_{t\epsilon}}g(x)^{p-r}\left(L_{w,\epsilon}^{t}f(x)-\sigma_{1}^{t}\mcal{L}_{w}^{t}f(x)\right)^2dx\\&\lesssim M^{2}.
\end{align*}
Furthermore, near the boundary, according to Lemma \ref{lemma64} and Lemma \ref{lemma62boundary}, it yields that
\begin{align*}
    E_{w,\epsilon}(L_{w,\epsilon}^{t}f,\partial_{t\epsilon} \mcal{X})\lesssim \frac{1}{\epsilon^{2}}\|L_{w,\epsilon}^{t}f\|_{L^{2}(\partial_{t\epsilon}\mcal{X},g^{p+r})}\lesssim M^2.
\end{align*}

Putting all pieces above together, we obtain the proof for the case when $s$ is odd and $t:=(s-1)/2$. Similar arguments can be applied to the case when $s$ is even and $t:=(s-2)/2$. Therefore, combining all above together, we obtain for all integer $s>1$:
\begin{align*}
    \langle L_{w,\epsilon}^{s}f,f \rangle_{g^{p-r}}\lesssim M^{2}.
\end{align*}
\end{proof}

We are now in the position to prove the main results of Section \ref{mainthm}.
\begin{proof}[Proof of Theorem \ref{firstorder}]
By Cauchy-Schwarz inequality, we have: for all $s\in\mbb{N}_{+}$:
\begin{align*}
    \|\hat{f}-f\|_{w,n}^2\le 2(\|\mbb{E}\hat{f}-f|_{w,n}^2+\|\hat{f}-\mbb{E}\hat{f}\|_{w,n}^2).
\end{align*}
Then, according to PCR-WLE algorithm in Section \ref{modelandalg}, we obtain
\begin{align}\label{decomp:bias}
    \|\mbb{E}\hat{f}-f\|_{w,n}^{2}=\sum_{k=K+1}^{n}\langle v_k,f \rangle_{w,n}^2\le \frac{\langle L_{w,n,\epsilon}^{s}f,f\rangle_{w,n}}{\lambda_{K+1}^{s}},
\end{align}
and
\begin{align*}
    \|\hat{f}-\mbb{E}\hat{f}\|_{w,n}^2=\sum_{k=1}^{K}\langle v_k,\varepsilon\rangle_{w,n}^{2}.
\end{align*}
Since $\langle v_k,\varepsilon\rangle_{w,n}$ is normally distributed with $0$ mean and variance:
\begin{align}\label{var1}
    \Var \langle v_k,\varepsilon\rangle_{w,n}=\frac{1}{n^2}v_k^TD^{\frac{2(p-1-r)}{q-1}}v_k,
\end{align}
where $\langle v_k,v_k\rangle_{w,n}=\frac{1}{n}v_{k}^{T}D^{\frac{p-1-r}{q-1}}v_k=1$. Note that $\langle v_k/\sqrt{n},v_k/\sqrt{n}\rangle_{g^{p-r}}=1$, then we have
\begin{align}\label{eigenD}
    \underset{v_{k}/\sqrt{n}\in\mbb{R}^{n}}{\min}~\frac{1}{n}v_k^{T}D^{\frac{p-1-r}{q-1}}D^{\frac{p-1-r}{q-1}}v_{k}
\end{align}
is the smallest eigenvalue of the matrix $D^{\frac{p-1-r}{q-1}}$ with respect to the inner product $\langle\cdot,\cdot\rangle_{g^{p-r}}$. As $D$ is a diagonal matrix with the $(i,i)$-element as $d_i$, according to Section \ref{kdesectionpf}, it is bounded from below, say by a constant $C>0$, almost surly for $n$ large enough. Then, combining \eqref{var1} and \eqref{eigenD}, we have:
\begin{align*}
    \|\hat{f}-\mbb{E}\hat{f}\|_{w,n}^2=\frac{1}{n}\sum_{k=1}^{K}(\sqrt{n}\langle v_k,\epsilon\rangle_{w,n})^{2},
\end{align*}
with $\sqrt{n}\langle v_k,\epsilon\rangle_{w,n}$ being normal with mean $0$ and variance 
\begin{align*}
    \Var(\sqrt{n} \langle v_k,\epsilon\rangle_{w,n})\ge C>0.
\end{align*}
According to an exponential inequality for chi-square distributions from \cite{laurent2000adaptive}, we obtain:
\begin{align}\label{decomp:variance}
    \mbb{P}\left(\|\hat{f}-\mbb{E}\hat{f}\|_{w,n}^2\ge \frac{CK}{n}+2\frac{\sqrt{K}}{n}\sqrt{t}+2\frac{t}{n}\right)\le e^{-t}.
\end{align}
With \eqref{decomp:bias} and \eqref{decomp:variance}, it yields
\begin{align}\label{biasvariancedecomp}
    \|\hat{f}-f\|_{w,n}^{2}\le \frac{\langle L_{w,n,\epsilon}^{s}f,f\rangle_{w,n}}{\lambda_{K+1}^{s}}+\frac{CK}{n},
\end{align}
with probability at least $1-e^{-K}$ if $1\le K\le n$. Then, it remains to bound the empirical weighted Sobolev seminorm $\langle L_{w,n,\epsilon}^{s}f,f\rangle_{w,n}$ and the graph weighted Laplacian eigenvalue $\lambda_{K+1}^{s}$.

We will first focus on $\langle L_{w,n,\epsilon}^{s}f,f\rangle_{w,n}$ for $s=1$. By definition \eqref{def:weightedL}, we have by symmetry:
\begin{align}\label{disexpress}
    \mbb{E}\langle L_{w,n,\epsilon}f,f\rangle_{w,n}=\frac{1}{2}\mbb{E}\left(\frac{1}{\epsilon^{d+2}}|d_{i}^{-\frac{r}{q-1}}f(X_i)-d_{j}^{-\frac{r}{q-1}}f(X_j)|^{2}d_{i}^{\frac{1-p}{q-1}}\frac{\eta\left(\frac{\|X_i-X_{j}\|}{\epsilon}\right)}{\tilde{d}_{i}^{1-q/2}\tilde{d}_{j}^{1-q/2}}\right).
\end{align}
We would like to point out here that the normalization factor $\epsilon^{-(d+2)}$ is motivated by the fact that a factor of $\epsilon^{-d}$ is needed to scale $\eta\left(\frac{\|X_i-X_j\|}{\epsilon}\right)$ and the remaining factor, $\epsilon^{-2},$ stabilized the squared differences of $d_{i}^{-\frac{r}{q-1}}$ under the expectation.

According to Section \ref{kdesectionpf} and by conditioning on $X_i$ and the law of iterated expectation, we have for $n$ large enough,
\begin{align}\label{distononlocal}
\resizebox{0.95\hsize}{!}{$\left|\mbb{E}\left(\frac{1}{\epsilon^{d+2}}|d_{i}^{-\frac{r}{q-1}}f(X_i)-d_{j}^{-\frac{r}{q-1}}f(X_j)|^{2}d_{i}^{\frac{1-p}{q-1}}\frac{\eta\left(\frac{\|X_i-X_{j}\|}{\epsilon}\right)}{\tilde{d}_{i}^{1-q/2}\tilde{d}_{j}^{1-q/2}}\right)-2\langle L_{w,\epsilon}f,f\rangle_{g^{p-r}}\right| \lesssim \Delta(n,\epsilon,\eta,g)+\epsilon,$}
\end{align}
where 
\begin{align*}
    \Delta(n,\epsilon,\eta,g): =\frac{1}{n}g_{\max}+\frac{\eta(0)}{n\epsilon^{d}}+\frac{n-1}{n}\Big(\sqrt{\frac{|\log \epsilon|}{n\epsilon^{d}}}+\epsilon\Big) \to 0\quad\text{as }\;n\to \infty.
\end{align*}

Combining \eqref{disexpress}, \eqref{distononlocal} and Lemma \ref{lemma61}, we obtain:
\begin{align*}
    \mbb{E}\langle L_{w,n,\epsilon}f,f\rangle_{w,n}\lesssim M^{2}+\Delta(n,\epsilon,\eta,g)+\epsilon.
\end{align*}
Consequently, by Markov's inequality, we have: for any $\delta\in(0,1)$,
\begin{align}\label{markovs=1}
    \langle L_{w,n,\epsilon}f,f\rangle_{w,n}\lesssim \frac{1}{\delta}\left(M^{2}+\Delta(n,\epsilon,\eta,g)+\epsilon\right),
\end{align}
with probability at least $1-\delta$. Note that the above bound on the expected weighted Sobolev seminorm generalizes the results in \cite{green2021minimax} to the weighted Laplacians by some properties of KDE.

Next, we proceed to the higher order case when $s>1$ for $\langle L_{w,n,\epsilon}^{s}f,f\rangle_{w,n}$. We define the following difference operator:
\begin{align*}
    D_{j}f(x)=(d_{\cdot}^{-\frac{r}{q-1}}f(x)-d_{j}^{-\frac{r}{q-1}}f(X_j))d_{\cdot}^{\frac{1-p}{q-1}}w_{\cdot,j}^{\epsilon},
\end{align*}
where $d_{\cdot}$ and $w_{\cdot,j}^{\epsilon}$ are defined by replacing $X_i$ by 
$x$ in both $d_{i}$ and $w_{i,j}^{\epsilon}.$ Furthermore, let $D_{\mathbf{j}}f(x):=(D_{j_1}f\circ\ldots\circ D_{j_s}f)(x)$, where $\textbf{j} = (j_1,\ldots,j_s) \in [n]^s := \{1,\ldots,n\}^s$. Denote by $(n)^{s} $ the sub-collection of vectors in $[n]^s$ with no repeated indices and let by $i\mbf{j}\coloneqq (i,j_1,\ldots,j_{s}).$

Following the idea of \cite[Proof of Lemma 3]{green2021minimax}, we decompose the weighted Sobolev seminorm into a U-statistic, which is an unbiased estimator of the non-local Sobolev seminorm $\langle L_{w,\epsilon}^{s}f,f\rangle_{g^{p-r}}$, and a pure bias term:
\begin{align}\label{I1I2}
    \langle L_{w,n,\epsilon}^{s}f,f\rangle_{w,n}&=\frac{1}{n}\sum_{i=1}^{n}d_{i}^{\frac{p-1-r}{q-1}}L_{w,n,\epsilon}^{s}f(X_i)\cdot f(X_i)\nonumber\\&=\frac{1}{n\epsilon^{2s}}\sum_{i\mbf{j}\in (n)^{s+1}}d_{i}^{\frac{p-1-r}{q-1}}D_{\mbf{j}}f(X_i)\cdot f(X_i)\nonumber\\&\qquad+\frac{1}{n\epsilon^{2s}}\sum_{i\mbf{j}\in [n]^{s+1}\backslash(n)^{s+1}}d_{i}^{\frac{p-1-r}{q-1}}D_{\mbf{j}}f(X_i)\cdot f(X_i)\nonumber\\&=:I_1+I_2.
\end{align}

\textcolor{black}{Note that there are errors in \cite[Proof of Lemma 3]{green2021minimax} when bounding both $\mbb{E}I_{1}$ and $\mbb{E}I_{2}$. Specifically, in \cite[Lemma D.3]{green2021minimax}, there should not be a $\delta$ appearing in Equation D.4 by Markov's inequality and the power of $\epsilon$ should be $2s+d$.} Although their final result is correct, we will fix these errors in the following proof. Now, determined by whether all $i\mbf{j}$ are distinct, the empirical weighted Sobolev seminorm can be divided into two parts, $I_1$ and $I_2$. The first one involves all distinct indices where we make approximation by the so-called non-local weighted sobolev norm $\langle L_{w,\epsilon}^{s}f,f\rangle_{g^{p-r}}$; the second part focuses on the case where not all $i\mbf{j}$ are distinct and use the fact that it is related to a connected subgraph.

As for $I_1$ from \eqref{I1I2}, we have
\begin{align*}
    \mbb{E}I_1&=\frac{1}{n\epsilon^{2s}}\frac{n!}{(n-s-1)!}\mbb{E}\left(d_{i}^{\frac{p-1-r}{q-1}}D_{\mbf{j}}f(X_i)\cdot f(X_i)\right)\\&=\frac{1}{n\epsilon^{2s}}\frac{n!}{(n-s-1)!}\mbb{E}\langle D_{\mbf{j}}f(X_i),f(X_i)\rangle_{g^{p-r}},
\end{align*}
where the operator $D_{\mbf{j}}$ is iterated for $s$ different times due to the fact that $i\mbf{j}$ are all distinct. For each iteration, say $s=1$, we have
\begin{align}\label{alldiffers1}
&\quad\mbb{E}\langle D_{j}f(X_i),f(X_i)\rangle_{g^{p-r}}\nonumber\\&=\frac{\epsilon^{2}}{2n}\mbb{E}\left(\frac{1}{\epsilon^{d+2}}|d_{i}^{-\frac{r}{q-1}}f(X_i)-d_{j}^{-\frac{r}{q-1}}f(X_j)|^{2}d_{i}^{\frac{1-p}{q-1}}\frac{\eta\left(\frac{\|X_i-X_{j}\|}{\epsilon}\right)}{\tilde{d}_{i}^{1-q/2}\tilde{d}_{j}^{1-q/2}}\right).
\end{align}
Then, plugging \eqref{distononlocal} in \eqref{alldiffers1}, we obtain
\begin{align*}
    \left|\mbb{E}\langle D_{j}f(X_i),f(X_i)\rangle_{g^{p-r}}-\frac{\epsilon^{2}}{n}\langle L_{w,\epsilon}f,f\rangle_{g^{p-r}}\right|\lesssim \frac{\epsilon^{2}}{2n}(\Delta(n,\epsilon,\eta,g)+\epsilon).
\end{align*}
After $s$ times iteration, it yields that
\begin{align*}
    \left|\mbb{E}\langle D_{\mbf{j}}f(X_i),f(X_i)\rangle_{g^{p-r}}-\frac{\epsilon^{2s}}{n^{s}}\langle L_{w,\epsilon}^{s}f,f\rangle_{g^{p-r}}\right|\lesssim \frac{\epsilon^{2s}}{2^{s}n^{s}}(\Delta(n,\epsilon,\eta,g)+\epsilon).
\end{align*}
Putting all above results back in $\mbb{E}I_1$, we conclude that for $n$ large enough,
\begin{align}\label{beforestirling}
    \left|\mbb{E}I_1-\frac{n!}{n^{s+1}(n-s-1)!}\langle L_{w,\epsilon}^{s}f,f \rangle_{g^{p-r}}\right|\lesssim \frac{n!}{n^{s+1}(n-s-1)!}(\Delta(n,\epsilon,\eta,g)+\epsilon).
\end{align}
The Stirling’s formula shows
\begin{align*}
    \underset{n\rightarrow\infty}{\lim}~\frac{n!}{n^{s+1}(n-s-1)!}=1.
\end{align*}
Therefore, by \eqref{beforestirling}, we have for $n$ large enough,
\begin{align*}
    \mbb{E}I_{1}\lesssim \langle L_{w,\epsilon}^{s}f,f \rangle_{g^{p-r}}+(\Delta(n,\epsilon,\eta,g)+\epsilon).
\end{align*}
According to Lemma \ref{lemma62}, it yields that
\begin{align}\label{ExpectI1}
    \mbb{E}I_{1}\lesssim M^{2}+(\Delta(n,\epsilon,\eta,g)+\epsilon).
\end{align}

We next shift our attention to $I_{2}$ in \eqref{I1I2}:
\begin{align*}
    \frac{1}{n\epsilon^{2s}}\sum_{i\mbf{j}\in [n]^{s+1}\backslash(n)^{s+1}}d_{i}^{\frac{p-1-r}{q-1}}D_{\mbf{j}}f(X_i)\cdot (f(X_i)-f(X_{j_1})).
\end{align*}

For $i\mbf{j}$ not all distinctive, if they contains a total of $(k+1)$ distinct indices for example for $1\le k\le s-1$, we have by symmetry:
\begin{align*}
    \sum_{i\mbf{j}\in [n]^{s+1}\backslash(n)^{s+1}}d_{i}^{\frac{p-1-r}{q-1}}D_{\mbf{j}}f(X_i)\cdot f(X_i)=\frac{1}{2}\cdot \sum_{i\mbf{j}\in [n]^{s+1}\backslash(n)^{s+1}}d_{i}^{\frac{p-1-r}{q-1}}D_{\mbf{j}}f(X_i)\cdot (f(X_i)-f(X_{j_1})).
\end{align*}
Observe that in order for 
\begin{align*}
    d_{i}^{\frac{p-1-r}{q-1}}|D_{\mbf{j}}f(X_i)|\cdot |f(X_i)-f(X_{j_1})|
\end{align*}
to be non-zero, it must be the case that the graph $G_{n,\epsilon}(X_{i\mbf{j}})$ which is the subgraph induced by the vertices $X_i,X_{j_1},\ldots,X_{j_{s}}$ is complete. Since we have:
\begin{align*}
    D_{ij}f(x)&=D_{i}(D_{j}f(x))\\&=D_{i}\left((d_{\cdot}^{-\frac{r}{q-1}}f(x)-d_{j}^{-\frac{r}{q-1}}f(X_j))d_{\cdot}^{\frac{1-p}{q-1}}w_{\cdot,j}^{\epsilon}\right)\\&=(d_{\cdot}^{-\frac{r}{q-1}}D_j f(x)-d_{i}^{-\frac{r}{q-1}}D_{j}f(X_i))d_{\cdot}^{\frac{1-p}{q-1}}w_{\cdot,i}^{\epsilon},
\end{align*}
then
\begin{align*}
    |D_{j_{1}j_{2}}f(X_{i})|\le \left(d_{i}^{-\frac{r}{q-1}}|D_{j_2}f(X_{i})|+d_{j_1}^{-\frac{r}{q-1}}|D_{j_2}f(X_{j_1})|\right)d_{i}^{\frac{1-p}{q-1}}w_{i,j_1}^{\epsilon}.
\end{align*}
Repeating the above computation and by induction, it yields that for $s\ge 2$,
\begin{align*}
    |D_{\mbf{j}}f(X_i)|\le (s-1)d_{\max/\min}^{-\frac{(s-1)r}{q-1}}d_{\max/\min}^{\frac{(s-1)(1-p)}{q-1}}\left(w_{\max}^{\epsilon}\right)^{s-1}\sum_{j\in i\mbf{j}\backslash\{j_{s}\}}|D_{j_{s}}f(X_{j})|,
\end{align*}
where $d_{\max}:=\underset{i=1,\ldots,n}{\max}~d_{i}$, $d_{\min}:=\underset{i=1,\ldots,n}{\min}~d_{i}$, $w_{\max}:=\underset{i,j=1,\ldots,n}{\max}~w_{i,j}$ and $d_{\max/\min}$ means it is $d_{\max}$ if $-(s-1)r/(d-1)$ (respectively $(s-1)(1-p)/(q-1)$) are positive and it is $d_{\min}$ otherwise.

According to Section \ref{kdesectionpf}, we have for $n$ large enough, $d_{\max}$ is bounded from above and $d_{\min}$ is bounded from below a.s. and 
\begin{align*}
    w_{\max}\lesssim \frac{1}{n\epsilon^{d}},
\end{align*}
almost surely.

Consequently, it yields that
\begin{align}\label{iteration}
    &d_{i}^{\frac{p-1-r}{q-1}}|D_{\mbf{j}}f(X_i)|\cdot |f(X_i)-f(X_{j_1})|\nonumber\\&=d_{i}^{\frac{p-1-r}{q-1}}|D_{\mbf{j}}f(X_i)|\cdot |f(X_i)-f(X_{j_1})|\cdot\mbf{1}_{\{G_{n,\epsilon}(X_{i\mbf{j}})\ \text{is connected}\}}\nonumber\\&\lesssim  \frac{1}{(n\epsilon^{d})^{s-1}}\sum_{j\in i\mbf{j}\backslash\{j_s\}}\left(d_{i}^{\frac{p-1-r}{q-1}}|D_{j_{s}}f(X_{j})|\cdot|f(X_i)-f(X_{j_1})|\cdot\mbf{1}_{\{G_{n,\epsilon}(X_{i\mbf{j}})\ \text{is connected}\}}\right)\nonumber\\&=\frac{\epsilon^{2}}{n^{s}\epsilon^{d(s-1)}}\sum_{j\in i\mbf{j}\backslash\{j_s\}}\Bigg(\frac{1}{\epsilon^{d+2}}d_{i}^{\frac{p-1-r}{q-1}}|d_{j}^{-\frac{r}{q-1}}f(X_j)-d_{j_s}^{-\frac{r}{q-1}}f(X_{j_{s}})|d_{j}^{\frac{1-p}{q-1}}\frac{\eta\left(\frac{\|X_j-X_{j_s}\|}{\epsilon}\right)}{\tilde{d}_{j}^{1-q/2}\tilde{d}_{j_{s}}^{1-q/2}}\nonumber\\&\qquad|f(X_i)-f(X_{j_1})| \mbf{1}_{\{G_{n,\epsilon}(X_{i\mbf{j}})\ \text{is connected}\}}\Bigg),
\end{align}
where we again assign $\epsilon^{d+2}$ as a normalization factor into the expectation as \eqref{disexpress}. 

Now, note that for $j=i$ in the summand on the right-hand side of \eqref{iteration}, we have according to Section \ref{kdesectionpf}:
\begin{align}\label{j=i}
    &\mbb{E}\Bigg(\frac{1}{\epsilon^{d+2}}d_{i}^{-\frac{r}{q-1}}|d_{i}^{-\frac{r}{q-1}}f(X_i)-d_{j_s}^{-\frac{r}{q-1}}f(X_{j_{s}})|\frac{\eta\left(\frac{\|X_i-X_{j_s}\|}{\epsilon}\right)}{\tilde{d}_{i}^{1-q/2}\tilde{d}_{j_{s}}^{1-q/2}}|f(X_i)-f(X_{j_1})|\nonumber\\&\qquad \mbf{1}_{\{G_{n,\epsilon}(X_{i\mbf{j}})\ \text{is connected}\}}\Bigg)\nonumber\\&\lesssim\mbb{E}\left(\left(\frac{1}{\epsilon^{d+2}}|d_{i}^{-\frac{r}{q-1}}f(X_i)-d_{j_s}^{-\frac{r}{q-1}}f(X_{j_{s}})|\frac{\eta\left(\frac{\|X_i-X_{j_s}\|}{\epsilon}\right)}{\tilde{d}_{i}^{1-q/2}\tilde{d}_{j_{s}}^{1-q/2}}|d_{i}^{-\frac{r}{q-1}}f(X_i)-d_{j_{1}}^{-\frac{r}{q-1}}f(X_{j_{1}})|\right.\right.\nonumber\\&\qquad+(\Delta(n,\epsilon,\eta,g)+\epsilon)\Bigg)\mbf{1}_{\{G_{n,\epsilon}(X_{i\mbf{j}})\ \text{is connected}\}}\Bigg)\nonumber\\&\lesssim \mbb{E}\left(\left(\frac{1}{\epsilon^{d+2}}|d_{i}^{-\frac{r}{q-1}}f(X_i)-d_{j_s}^{-\frac{r}{q-1}}f(X_{j_{s}})|^{2}\frac{\eta\left(\frac{\|X_i-X_{j_s}\|}{\epsilon}\right)}{\tilde{d}_{i}^{1-q/2}\tilde{d}_{j_{s}}^{1-q/2}}+(\Delta(n,\epsilon,\eta,g)+\epsilon)\right)\right.\nonumber\\&\qquad\mbf{1}_{\{G_{n,\epsilon}(X_{i\mbf{j}})\ \text{is connected}\}}\Bigg),
\end{align}
where the last inequality is by Cauchy–Schwarz inequality and $X_1,\ldots,X_{n}$ being i.i.d. data. Then, by integrating out all indices in $\mbf{j}$ not equal to $i$ or $j_{s}$, it yields that
\begin{align}\label{k+1}
    &\mbb{E}\left(\left(\frac{1}{\epsilon^{d+2}}|d_{i}^{-\frac{r}{q-1}}f(X_i)-d_{j_s}^{-\frac{r}{q-1}}f(X_{j_{s}})|^{2}\frac{\eta\left(\frac{\|X_i-X_{j_s}\|}{\epsilon}\right)}{\tilde{d}_{i}^{1-q/2}\tilde{d}_{j_{s}}^{1-q/2}}+(\Delta(n,\epsilon,\eta,g)+\epsilon)\right)\right.\nonumber\\&\qquad\mbf{1}_{\{G_{n,\epsilon}(X_{i\mbf{j}})\ \text{is connected}\}}\Bigg)\nonumber\\&\lesssim \left(C\epsilon^{d} g_{\max}V_{d}\right)^{k-1}\mbb{E}\left(\left(\frac{1}{\epsilon^{d+2}}|d_{i}^{-\frac{r}{q-1}}f(X_i)-d_{j_s}^{-\frac{r}{q-1}}f(X_{j_{s}})|^{2}\frac{\eta\left(\frac{\|X_i-X_{j_s}\|}{\epsilon}\right)}{\tilde{d}_{i}^{1-q/2}\tilde{d}_{j_{s}}^{1-q/2}}\right.\right.\nonumber\\&\qquad+(\Delta(n,\epsilon,\eta,g)+\epsilon)\Bigg)\Bigg).
\end{align}

Therefore, according to \eqref{j=i}, \eqref{k+1}, \eqref{distononlocal} and Lemma \ref{lemma61}, we obtain
\begin{align}\label{jifinal}
    &\mbb{E}\Bigg(\frac{1}{\epsilon^{d+2}}d_{i}^{-\frac{r}{q-1}}|d_{i}^{-\frac{r}{q-1}}f(X_i)-d_{j_s}^{-\frac{r}{q-1}}f(X_{j_{s}})|\frac{\eta\left(\frac{\|X_i-X_{j_s}\|}{\epsilon}\right)}{\tilde{d}_{i}^{1-q/2}\tilde{d}_{j_{s}}^{1-q/2}}|f(X_i)-f(X_{j_1})|\nonumber\\&\qquad \mbf{1}_{\{G_{n,\epsilon}(X_{i\mbf{j}})\ \text{is connected}\}}\Bigg)\nonumber\\&\lesssim \epsilon^{d(k-1)}\left(M^{2}+\Delta(n,\epsilon,\eta,g)+\epsilon\right).
\end{align}
Applying a similar approach to all $j\neq j_{s}$ and plugging \eqref{jifinal} in \eqref{iteration} and \eqref{I1I2}, we have
\begin{align*}
    \mbb{E}I_{2}&\lesssim \frac{1}{n\epsilon^{2s}}\frac{1}{n^{s}\epsilon^{d(s-1)}}\sum_{k=1}^{s-1}\epsilon^{d(k-1)}\left(M^{2}+\Delta(n,\epsilon,\eta,g)\right)n^{k+1}\\&\lesssim \frac{\epsilon^{2}}{n\epsilon^{2s}}\left(M^{2}+\Delta(n,\epsilon,\eta,g)+\epsilon\right)\sum_{k=1}^{s-1}\frac{(n\epsilon^{d})^{k}}{(n\epsilon^{d})^{s}}n.
\end{align*}
Note that the above sum is bounded from above when $k=s-1$ by the assumption $n\epsilon^{d}\ge 1$. Finally, we conclude that
\begin{align}\label{ExpectI2}
    \mbb{E}I_{2}\lesssim \frac{\epsilon^{2}}{n\epsilon^{2s+d}}\left(M^{2}+\Delta(n,\epsilon,\eta,g)+\epsilon\right).
\end{align}

Finally, combining \eqref{I1I2}, \eqref{ExpectI1} and \eqref{ExpectI2}, we obtain:
\begin{align*}
    \mbb{E}\langle L_{w,n,\epsilon}^{s}f,f\rangle_{w,n}&\lesssim M^{2}+(\Delta(n,\epsilon,\eta,g)+\epsilon)+\frac{\epsilon^{2}}{n\epsilon^{2s+d}}\left(M^{2}+\Delta(n,\epsilon,\eta,g)+\epsilon\right)\\&\lesssim M^{2}+(\Delta(n,\epsilon,\eta,g)+\epsilon),
\end{align*}
where the last step is by the assumption that $\epsilon\gtrsim  n^{-1/(2(s-1)+d)}$. By Markov's inequality, we have for any $\delta\in (0,1)$,
\begin{align}\label{markovs>1}
    \langle L_{w,n,\epsilon}^{s}f,f\rangle_{w,n}\lesssim \frac{1}{\delta}\left(M^{2}+(\Delta(n,\epsilon,\eta,g)+\epsilon\right),
\end{align}
with probability at least $1-2\delta$. This bound can be considered as a higher order variant of \eqref{markovs=1} for $s>1$.

Now, recall the bound \eqref{biasvariancedecomp}. We have bounded the empirical weighted Sobolev seminorm by \eqref{markovs=1} and \eqref{markovs>1}. It remains to bound the eigenvalues $\lambda_{K+1}$.

According to Lemma \ref{eigenref}, we have:
\begin{align}\label{eigenboundreadytouse}
    \lambda_{k}=\lambda_{k}(L_{w,n,\epsilon})\gtrsim \lambda_{k}(\mcal{L}_{w}) \wedge \epsilon^{2},\ \text{for all}\ 2\le k\le n,
\end{align}
with probability at least $1-Cne^{-cn\epsilon^{d}}$ for some constants $C,c>0$.

For $s=1$, combining \eqref{biasvariancedecomp}, \eqref{markovs=1} and \eqref{eigenboundreadytouse}, we have with probability at least $1-\delta-Cne^{-cn\epsilon^{d}}-e^{-K}$ and $n$ large enough:
\begin{align*}
    \|\hat{f}-f\|_{w,n}^{2}\lesssim \frac{M^{2}}{\delta\left(\lambda_{K+1}(\mcal{L}_{w}) \wedge \epsilon^{2}\right)}+\frac{K}{n}.
\end{align*}
Furthermore, based on the assumption $\epsilon\lesssim K^{-1/d}$ and Proposition \ref{refWeyl}, the above inequality becomes:
\begin{align}\label{tradeoffres}
    \|\hat{f}-f\|_{w,n}^{2}\lesssim \frac{M^{2}}{\delta}(K+1)^{-2/d}+\frac{K}{n}.
\end{align}
By balancing the two terms on the right-hand side, we pick $K=\lfloor M^2n \rfloor^{d/(2+d)}$. Then, it yields that
\begin{align}\label{finallyminimaxopt}
    \|\hat{f}-f\|_{w,n}^{2}\lesssim \frac{1}{\delta}M^2(M^2n)^{-2/(2+d)}.
\end{align}
If $M^2<n^{-1}$, we can take $K=1$ and obtain from \eqref{tradeoffres} that:
\begin{align*}
\|\hat{f}-f\|_{w,n}^{2}\lesssim \frac{1}{n\delta}.
\end{align*}
If $M>n^{1/d}$, we take $K=n$ and in this case, we actually have $\hat{f}(X_i)=Y_i$ for $i=1,\ldots,n$ and
\begin{align*}
    \|\hat{f}-f\|_{w,n}^{2}=\frac{1}{n}\sum_{i=1}^{n}\varepsilon_{i}^{2}\lesssim C,
\end{align*}
with probability at least $1-e^{-n}$ for some constant $C$. Combining all above cases depending on choices of $K$, it yields that bound in Theorem \ref{firstorder} .

For $s>1$, the proof follows in a similar way by considering \eqref{markovs>1} instead of \eqref{markovs=1}.
\end{proof}

\subsection{Proof of Theorem \ref{Lepski}}
\begin{proof}[Proof of Theorem \ref{Lepski}]
Recall the construction of the estimator based on Lepski's procedure: $\hat{f}_{\adapt}=\hat{f}_{\hat{s},\hat{M}}$ with $\hat{s},\hat{M}$ given in Section \ref{seclepski}. Let the event $\mcal{E}_{j}$ be that $\hat{s}=s_{j}$ and suppose $s=s_{i}$ for the true smooth parameter. 

First of all, it suffices to consider $M\in \mcal{D}$ by realizing that if $M\in (M_{j-1},M_{j})$, then $f\in H^{s}(\mcal{X},g;M)$ with $H^{s}(\mcal{X},g;M_{j-1})\subset H^{s}(\mcal{X},g;M)\subset H^{s}(\mcal{X},g;M_{j})$. Now, we also suppose $M=M_i$ correspondingly and consider bounding the sum:
\begin{align*}
    \sum_{j=1}^{N_l}\left(\|\hat{f}_{s_j}-f\|_{w,n}^{2}M_i^{-2}(M_i^2 n/\log n )^{2s_i/(2s_i+d)}\mbf{1}_{\mcal{E}_{j}}\right),
\end{align*}
conditional on the event that the sample points $X_1,\ldots,X_n$ satisfy \eqref{biasvariancedecomp} and \eqref{eigenboundreadytouse} with $K=\lfloor M_i^2n \rfloor^{d/(2s_i+d)}$. These two statements hold with probability at least $1-Cne^{-Cn\epsilon^{d}}-e^{-\lfloor M_i^2n \rfloor^{d/(2s_i+d)}}$. As we will see, the fact that this sum does not explode, relies on the fact that the probabilities of the sets ${\mcal E}_j$ get small as $n \to \infty$.

First, note that by Cauchy-Schwarz inequality, we have
\begin{align*}
    &\sum_{j=i}^{N_l}\left(\|\hat{f}_{s_j}-f\|_{w,n}^{2}M_i^{-2}(M_i^2 n/\log n )^{2s_i/(2s_i+d)}\mbf{1}_{\mcal{E}_{j}}\right)\\&\le \sum_{j=i}^{N_l}\left(\|\hat{f}_{s_j}-\hat{f}_{s_i}+\hat{f}_{s_j}-f\|_{w,n}^{2}M_i^{-2}(M_i^2 n/\log n )^{2s_i/(2s_i+d)}\mbf{1}_{\mcal{E}_{j}}\right)\\&\le \sum_{j=i}^{N_l}\left(2c_0^2 \mbf{1}_{\mcal{E}_{j}}+2\left(\|\hat{f}_{s_i}-f\|_{w,n}^{2}M_i^{-2}(M_i^2 n/\log n)^{2s_i/(2s_i+d)}\mbf{1}_{\mcal{E}_{j}}\right)\right)\\&\le 2c_0^2+2\left(\|\hat{f}_{s_i}-f\|_{w,n}^{2}M_i^{-2}(M_i^2 n/\log n)^{2s_i/(2s_i+d)}\right).
\end{align*}
Therefore, according to Theorem \ref{firstorder}, we have: for any $\delta\in(0,1)$,
\begin{align*}
    \sum_{j=i}^{N_l}\left(\|\hat{f}_{s_j}-f\|_{w,n}^{2}M_i^{-2}(M_i^2 n/\log n )^{2s_i/(2s_i+d)}\mbf{1}_{\mcal{E}_{j}}\right)\lesssim \frac{1}{\delta},
\end{align*}
with probability at least $1-\delta \log^{-2s_i/(2s_i+d)}n-Cne^{-Cn\epsilon^{d}}-e^{-\lfloor M_i^2n \rfloor^{d/(2s_i+d)}}$.

Next, we consider the other part when $j<i$:
\begin{align}\label{sumj<i1}
    \sum_{j=1}^{i-1}\left(\|\hat{f}_{s_j}-f\|_{w,n}^{2}M_i^{-2}(M_i^2 n/\log n )^{2s_i/(2s_i+d)}\mbf{1}_{\mcal{E}_{j}}\right).
\end{align}

By the definition, on the event $\mcal{E}_{j}$, there exists $s'\in \mcal{B}$ with $s'<s_i$ such that $\|\hat{f}_{s_i}-\hat{f}_{s'}\|_{w,n}>c_0 M'^{-2}(M'^2 n/\log n)^{-s'/(2s'+d)}$. This means $\|\hat{f}_{s_i}-\hat{f}_{s'}\|_{w,n}^2 M'^{-2}(M'^2 n/\log n)^{2s'/(2s'+d)}>c_0^{2}$. By triangle inequality, this implies we have either $\|\hat{f}_{s_i}-f\|_{w,n}^2 M'^{-2}(M'^2 n/\log n)^{2s'/(2s'+d)}>c_0^2/4$ or $\|\hat{f}_{s'}-f\|_{w,n}^2 M'^{-2}(M'^2 n/\log n)^{2s'/(2s'+d)}>c_0^2/4$. Then, we have
\begin{align}\label{PmalE}
    \mbb{P}(\mcal{E}_{j})&\le \sum_{l=1}^{i-1}\left(\mbb{P}\left(\|\hat{f}_{s_i}-f\|_{w,n}^2 M_l^{-2}(M_l^2 n/\log n)^{2s_l/(2s_l+d)}>c_0^2/4\right)\right.\nonumber\\&\quad\left.+\mbb{P}\left(\|\hat{f}_{s_l}-f\|_{w,n}^2 M_l^{-2}(M_l^2 n/\log n)^{2s_l/(2s_l+d)}>c_0^2/4\right)\right).
\end{align}
Since $l<i$, we have $f\in H^{s_i}(\mcal{X},g;M_l)\subset H^{s_{l}}(\mcal{X},g;M_l)$ for all $l<i$. Therefore, it suffices to focus on the concentration inequality of $\hat{f}_{s_l}$ to $f$, i.e., bounding
\begin{align}\label{j<1part1}
    \mbb{P}\left(\|\hat{f}_{s_l}-f\|_{w,n}^2 M_l^{-2}(M_l^2 n/\log n)^{2s_l/(2s_l+d)}>c_0^2/4\right).
\end{align}

Note that the key problem here is the rate of convergence of $\|\hat{f}_{s_j}-f\|_{w,n}^{2}$ in \eqref{sumj<i1} does not match the rate $(n/\log n)^{2s_i/(2s_i+d)}$ given there. However, this can be dealt with by controlling the probability of the event $\mcal{E}_{j}$. The strategy here is we need a better concentration inequality than what has been proven previously as \eqref{markovs>1} otherwise the probability of the event $\mcal{E}_{j}$ will not decay to $0$. Observe that the concentration \eqref{markovs>1}: for $n$ large enough and with probability smaller than $1-2\delta$,
\begin{align*}
    \langle L_{w,n,\epsilon}^{s}f,f\rangle_{w,n}\lesssim \delta^{-1} M^2,
\end{align*}
is from the application of Markov's inequality with
\begin{align*}
    \mbb{E}\langle L_{w,n,\epsilon}^{s}f,f\rangle_{w,n}\lesssim M^2,
\end{align*}
for $n$ large enough. While bounding the first moment gives a concentration inequality with probability $1-2\delta$, establishing a higher moment bound, e.g. the second moment, would result in a better concentration inequality with higher probability similar to \cite[Proposition 1]{green2021minimaxsmoothing}, which fits in our proof technique.

Starting with $s=1$ and similar to \eqref{distononlocal}, we have: for $n$ large enough,
\begin{align}\label{varianceconcs=1}
    &\Var \langle L_{w,n,\epsilon}f,f \rangle_{w,n}\nonumber\\\lesssim& \Var \left(\frac{1}{2}\frac{1}{n^2\epsilon^{d+2}}\sum_{i,j=1}^{n}(g(X_i)^{-r}f(X_i)-g(X_j)^{-r}f(X_j))^2g(X_{i})^{1-p}\frac{\eta\left(\frac{\|X_i-X_j\|}{\epsilon}\right)}{g(X_i)^{1-q/2}g(X_j)^{1-q/2}}\right).
\end{align}
For $i,j\in 1,\ldots,n$, let $$V_{ij}:=(g(X_i)^{-r}f(X_i)-g(X_j)^{-r}f(X_j))^2g(X_{i})^{1-p}\frac{\eta\left(\frac{\|X_i-X_j\|}{\epsilon}\right)}{g(X_i)^{1-q/2}g(X_j)^{1-q/2}}.$$

We have:
\begin{align*}
    &\Var \left(\sum_{i,j=1}^{n}(g(X_i)^{-r}f(X_i)-g(X_j)^{-r}f(X_j))^2g(X_{i})^{1-p}\frac{\eta\left(\frac{\|X_i-X_j\|}{\epsilon}\right)}{g(X_i)^{1-q/2}g(X_j)^{1-q/2}}\right)\\&=\sum_{i,j=1}^{n}\sum_{l,m=1}^{n}\text{Cov}(V_{ij},V_{lm}).
\end{align*}
Now, consider the following four scenarios depending on the cardinality of $\{i,j,l,m\}$.
\begin{itemize}
    \item If $|\{i,j,l,m\}|=4$, since $V_{ij}$ and $V_{lm}$ are independent, we have $\text{Cov}(V_{ij},V_{lm})=0$.
    \item If $|\{i,j,l,m\}|=3$, without loss of generality, say $i=l$, we have by Lipschitz condition,
    \begin{align*}
        \text{Cov}(V_{ij},V_{im})&\le \mbb{E}[V_{ij}V_{im}]\\&\lesssim \epsilon^{2d+4}M^{4}.
    \end{align*}
    \item If $|\{i,j,l,m\}|=2$, without loss of generality, say $i=l$ and $j=m$, similarly, we obtain 
    \begin{align*}
        \text{Cov}(V_{ij},V_{ij})&\le \mbb{E}V_{ij}^2\\&\lesssim \epsilon^{d+4}M^{4}.
    \end{align*}
    \item If $|\{i,j,l,m\}|=1$, we have $V_{ij}=V_{lm}=0$. 
\end{itemize}
Plugging the above results in \eqref{varianceconcs=1}, it yields that for $n$ large enough,
\begin{align*}
    \Var \langle L_{w,n,\epsilon}f,f \rangle_{w,n}\lesssim \frac{1}{4n^{4}\epsilon^{2d+4}}\left(n^{3}\epsilon^{2d+4}M^{4}+n^2\epsilon^{d+4}M^{4}\right)\lesssim n^{-1}M^4,
\end{align*}
where the last step follows by the assumption that $n\epsilon^{d}\ge 1$. Then, by Markov's inequality, we obtain: for any $\delta\in (0,1)$,
\begin{align}\label{markovvariances=1}
    \mbb{P}\left(\left|\langle L_{w,n,\epsilon}f,f \rangle_{w,n}-\mbb{E} \langle L_{w,n,\epsilon}f,f \rangle_{w,n}\right|\ge  \frac{1}{\delta} M^{2}\right)\lesssim \frac{\delta^{2}}{n}.
\end{align}
Combining \eqref{markovvariances=1} and \eqref{markovs=1}, we conclude that for $n$ large enough,
\begin{align*}
    \langle L_{w,n,\epsilon}f,f \rangle_{w,n}\lesssim \frac{1}{\delta}M^2
\end{align*}
holds with probability not less than $1-\frac{\delta}{n^2}$. Furthermore, following a similar argument in Lemma \ref{lemma62}, one can show the above high-probability bound also holds for the case $s>1$. Thus, under the additional Lipschitz assumption that $|f_{g}(x)-f_{g}(x')|\le M\|x-x'\|$, we establish a better bound for the empirical weighted Sobolev seminorm: for all $s\in\mbb{N}_{+}$ and $n$ large enough,
\begin{align*}
    \langle L_{w,n,\epsilon}^s f,f \rangle_{w,n}\lesssim \frac{1}{\delta}M^2,
\end{align*}
with probability at least $1-C\frac{\delta^2}{n}$.

Conditional on the event that the sample points $X_1,\ldots,X_n$ satisfy \eqref{biasvariancedecomp} and \eqref{eigenboundreadytouse} with $K=\lfloor M^2n \rfloor^{d/(2s+d)}$, following the proof of Theorem \ref{firstorder} to obtain \eqref{finallyminimaxopt} by using the better concentration inequality we derived above instead, we have for $n$ large enough,
\begin{align*}
    \|\hat{f}-f\|_{w,n}^{2}\lesssim \frac{1}{\delta}M^2(M^2n)^{-2s/(2s+d)},
\end{align*}
with probability at least $1-C\delta^2 n^{-1}-Cne^{-Cn\epsilon^{d}}-e^{-\lfloor M^2n \rfloor^{d/(2s+d)}}$ under the minimax optimal setting for $M$.

Now, returning to our mission \eqref{j<1part1}, by setting $\delta^{-1}=c_{0}^2/4\cdot \log^{2s_l/(2s_l+d)} n$, we have:
\begin{align*}
    &\mbb{P}\left(\|\hat{f}_{s_l}-f\|_{w,n}^2 M_l^{-2}(M_l^2 n/\log n)^{2s_l/(2s_l+d)}>c_0^2/4\right)\\&\le 16C c_0^{-4} n^{-1}\log^{-2s_l/(2s_l+d)} n+Cne^{-Cn\epsilon^{d}}+e^{-\lfloor M_{\min}^2 n \rfloor^{d/(2s+d)}}.
\end{align*}
With \eqref{PmalE}, we obtain:
\begin{align*}
    \mbb{P}(\mcal{E}_{j})\le 16C c_0^{-4} n^{-1}\log^{1-2s_{\min}/(2s_{\min}+d)} n+Cn e^{-Cn\epsilon^{d}}\log n +e^{-\lfloor M_{\min}^2 n \rfloor^{d/(2s+d)}}\log n .
\end{align*}
Combining the above result with \eqref{sumj<i1} and noting that on $\mcal{E}_{j}^{c}$, $\mbf{1}_{\mcal{E}_{j}}=0$, it yields that
\begin{align*}
    \sum_{j=1}^{i-1}\left(\|\hat{f}_{s_j}-f\|_{w,n}^{2}M_i^{-2}(M_i^2 n/\log n )^{2s_i/(2s_i+d)}\mbf{1}_{\mcal{E}_{j}}\right)\lesssim \frac{1}{\delta},
\end{align*}
with probability at least $$1-\delta \log^{-2s_i/(2s_i+d)}n-16C c_0^{-4} n^{-1}\log^{2-2s_{\min}/(2s_{\min}+d)} n-Cn e^{-Cn\epsilon^{d}}\log^2 n -e^{-\lfloor M_{\min}^2 n \rfloor^{d/(2s+d)}}\log^2 n.$$

\end{proof}

\section{Conclusion}

In this work, we provide adaptive and non-adaptive rates of convergence, in Theorem~\ref{firstorder} and~\ref{Lepski} respectively, for estimating a true regression function lying belonging to the Sobolev space. Our estimators are based on performing principal components regression based on the eigenvectors of the weighted graph Laplacian matrix, and using Lepski's method for deriving the adaptive results. Our contributions expand upon the non-adaptive outcome outlined in~\cite{green2021minimax}, which was originally established for a particular normalized graph Laplacian. This extension encompasses a broad spectrum of weighted Laplacian matrices commonly employed in practical applications, including the unnormalized Laplacian and the random walk Laplacian among them.

Future works include (i) relaxing the assumption that the density $g$ is bounded from below, (ii) developing confidence intervals for the estimators by establishing asymptotic normality results and developing related bootstrap procedures, and (iii) developing estimators that are instance-optimal in the sense of~\cite{hoffman2002random}, i.e., estimators that achieve the best possible rate for a given combination of the true regression function $f$ and the sampling density $g$ by adaptively picking the parameters $p,q$ and $r$ in the weighted graph Laplacian matrix.

\subsection*{Acknowledgement.} We gratefully acknowledge support for this project from the National Science Foundation via grant NSF-DMS-2053918.

\bibliographystyle{alpha}
\bibliography{example}

\appendix

\section{Auxiliary results}

In the subsequent two sections, we introduce some important properties of KDE and eigenvalues of the weighted Laplacian matrices $L_{w,n,\epsilon}$ and the weighted Laplacian operators $\mcal{L}_{w}$ used in the previous proof respectively.

\subsection{Property of kernel density estimation}\label{kdesectionpf}
Consider a Kernel density estimator (KDE) on $\mcal{X}$:
\begin{align*}
    g_{n}(x):=\frac{1}{n\epsilon^{d}}\sum_{j=1}^{n}\eta\left(\frac{\|x-X_j\|}{\epsilon}\right),
\end{align*}
where $\eta$ is a kernel function. 

In \cite{gine2002rates}, it has been proven that the above KDE satisfies the following almost sure convergence:
\begin{align*}
    \|g_n(x)-\mbb{E}g_{n}(x)\|_{\infty}=O_{a.s.}\left(\sqrt{\frac{|\log \epsilon|}{n\epsilon^{d}}}\right),
\end{align*}
given the assumption that the kernel $\eta$ satisfies the kernel VC-type condition \ref{a4} and see Remark \ref{kernelvc} for more details. 

As for the bias, it is well-known that there exists a boundary effect on KDE due to the fact that (with probability 1) all the samples lie in the support of the density. However, when we are far enough away from the boundary such that $B_{x}(\epsilon)\subset \mcal{X}$, we have
\begin{align*}
    \left|\mbb{E}g_{n}(x)-g(x)\right|&=\left|\int_{\mcal{X}}\frac{1}{\epsilon^d}\eta\left(\frac{\|x-y\|}{\epsilon}\right)g(y)dy-g(x)\right|\\&\le \int_{\|z\|\le 1}\eta(\|z\|)|g(x+\epsilon z)-g(x)|dz\\&\lesssim \epsilon \int_{\mbb{R}^{d}}\|z\|\eta(\|z\|)dz\lesssim \epsilon,
\end{align*}
where the last step is by the assumption that $g$ is Lipschitz. As a result, for such values of $x$,
\begin{align*}
    \|g_n(x)-g(x)\|_{\infty}=O_{a.s.}\left(\sqrt{\frac{|\log \epsilon|}{n\epsilon^{d}}}+\epsilon\right).
\end{align*}

When $x$ is near the boundary, i.e., $B_{x}(\epsilon)\not\subset \mcal{X}$, we have $X_i \in B_x(\epsilon)$ with probability less than $C\epsilon$ for some constant $C>0$.
Then:
\begin{align*}
    \mbb{E}g_{n}(x)\le g_{\max}\int_{\|z\|\le 1}\eta(\|z\|)dz<\infty,
\end{align*}
and
\begin{align*}
    \mbb{E}g_{n}(x)= \int_{\{\|z\|\le 1\}\cap\{x+\epsilon z\in\mcal{X}\} }\eta(\|z\|)g(x+\epsilon z)dz\ge g_{\min}\int_{\{\|z\|\le 1\}\cap\{x+\epsilon z\in\mcal{X}\} }\eta(\|z\|)dz>0,
\end{align*}
under the assumption \ref{a1} on $\mcal{X}$. Therefore, we have for all $x\in\mcal{X}$, $g_{n}(x)$ is bounded from above and below a.s. for $n$ large enough. 
 
By conditioning on $X_i$ and the law of total probability, we have for all $i\in[n]$ and $B_{\epsilon}(X_i)\in\mcal{X}$,
\begin{align*}
    \Delta^{-}(n,\epsilon,\eta,g)\le g_{n}(X_i)-g(X_{i})\le \Delta^{+}(n,\epsilon,\eta,g),
\end{align*}
almost surely with 
\begin{align*}
    \Delta^{-}(n,\epsilon,\eta,g)&:=-\frac{1}{n}g_{\max}+\frac{\eta(0)}{n\epsilon^{d}}-\frac{n-1}{n}\Delta(n,\epsilon),\\
    \Delta^{+}(n,\epsilon,\eta,g)&:=-\frac{1}{n}g_{\min}+\frac{\eta(0)}{n\epsilon^{d}}+\frac{n-1}{n}\Delta(n,\epsilon),
\end{align*}
and
\begin{align*}
    \Delta(n,\epsilon):=\sqrt{\frac{|\log \epsilon|}{n\epsilon^{d}}}+\epsilon.
\end{align*}

Since we are seeking a high-probability bound in Theorems \ref{firstorder}, it is not necessarily required to have an exact estimation near the boundary, which happens with probability of the order $\epsilon$. However, various approaches including data reflection, transformations, boundary kernels and local likelihood, have been proposed for boundary correction.

\subsection{Property of eigenvalues}
In this section, we focus on introducing some results on the eigenvalues of the weighted Laplacian $L_{n,w,\epsilon}$ and the weighted Laplacian operator $\mcal{L}_{w}$ based on analysis in \cite{calder2022improved,green2021minimaxsmoothing}.

\subsubsection{\capitalisewords{Transportation distance between measures}}
For a probability measure $G$ defined on $\mcal{X}$ and a map $T:\mcal{X}\rightarrow \mcal{X}$, denote by $T_{\sharp G}$ the push-forward of $G$ by $T$, i.e., the measure such that for any Borel subset $U\subseteq \mcal{X}$, it holds that
\begin{align*}
    T_{\sharp G}(U):=G(T^{-1}(U)).
\end{align*}
When $T_{\sharp G}$ is taken as the empirical measure of $G$ denoted by $G_{n}$, $T$ is called the transportation map between $G$ and $G_{n}$ and we define the $\infty$-transportation distance between $G$ and $G_{n}$ as 
\begin{align}\label{ot}
    d_{\infty}(G,G_{n}):=\underset{T:T_{\sharp G}=G_{n}}{\text{inf}}\|T-\text{Id}\|_{L^{\infty}(G)},
\end{align}
where $\text{Id}$ is the identity mapping. We denote by $\tilde{T}$ the optimal $\infty$-optimal transport map ($\infty$-OT map) between $G$ and $G_{n}$, i.e,, the map that achieves the infimum \eqref{ot}.

Now, following \cite{green2021minimaxsmoothing}, let
\begin{align*}
    \tilde{\delta}&:=\max\{n^{-1/d},C\epsilon\},
\end{align*}
where $C>0$ is some constant not depending on $n$ and we also let $\theta>0$ be some constant not depending on $n$. We present the following result from \cite{green2021minimaxsmoothing}.

\begin{Proposition}[cf. Proposition 3 of \cite{green2021minimaxsmoothing}]\label{transportdensity}
Under the assumptions \ref{a1} and \ref{a2}, with probability greater than $1-Cne^{-Cn\theta^{2} \tilde{\delta}^{d}}$, there exists a probability measure $\tilde{G}_{n}$ with density $\tilde{g}_{n}$ such that
\begin{align*}
    d_{\infty}(G_{n},\tilde{G}_{n})\le C\tilde{\delta},
\end{align*}
and such that
\begin{align*}
    \|g-\tilde{g}_{n}\|_{\infty}\le C(\theta+\tilde{\delta}),
\end{align*}
where $C>0$ is some constant not depending on $n$.
\end{Proposition}

\subsubsection{\capitalisewords{Discretization and interpolation maps}}
The key procedure adopted in \cite{calder2022improved} is to construct two maps: a discretization map $\tilde{\mcal{P}}:L^{2}(G)\rightarrow L^{2}(\tilde{G}_{n})$ and an interpolation map $\tilde{\mcal{I}}:L^{2}(\tilde{G}_{n})\rightarrow L^{2}(G)$, that are "almost" isometries.

For $X_{i}, i=1,\ldots,n$, define
\begin{align*}
    \tilde{U}_{i}:=\tilde{T}^{-1}(\{X_{i}\}).
\end{align*}
Then, we define the contractive discretization map $\tilde{\mcal{P}}:L^{2}(G)\rightarrow L^{2}(\tilde{G}_{n})$ by
\begin{align*}
    (\tilde{\mcal{P}}f)(X_i):=n\cdot \int_{\tilde{U}_{i}}f(x)\tilde{g}_{n}(x)dx.
\end{align*}
Moreover, the interpolation map $\tilde{\mcal{I}}:L^{2}(\tilde{G}_{n})\rightarrow L^{2}(G)$ is given by
\begin{align*}
    \tilde{\mcal{I}}u:=\Lambda_{\epsilon-2\tilde{\delta}}(\tilde{\mcal{P}}^{*}u).
\end{align*}
Here, $\tilde{\mcal{P}}^{*}=u\circ \tilde{T}$ is the adjoint of $\tilde{\mcal{P}}_{n}$, i.e., 
\begin{align*}
    (\tilde{\mcal{P}}^{*}u)(x)=\sum_{j=1}^{n}u(x_i)\mbf{1}_{x\in U_{i}},
\end{align*}
and $\Lambda_{\epsilon-2\tilde{\delta}}$ is a kernel smoothing operator with respect to a kernel $K$ (defined below) with the bandwidth $\epsilon-2\tilde{\delta}$. The kernel $K$ is defined by
\begin{align*}
    K(x,y):=\frac{1}{\epsilon^{d}}\zeta\left(\frac{\|x-y\|}{\epsilon}\right),
\end{align*}
where 
\begin{align*}
    \zeta(t):=\frac{1}{\sigma_1}\int_{t}^{\infty}\eta(s)sds.
\end{align*}
Then, define the operator $\Lambda_{h}$, for $h>0$, by
\begin{align*}
    \Lambda_{h}f(x):=\frac{1}{\tau(x)}\int_{\mcal{X}}K(x,y)f(y)g(y)dy,
\end{align*}
where $\tau(x):=\int_{\mcal{X}}K(x,y)g(y)dy$ is a normalization factor.

Furthermore, we define the Dirichlet energies:
\begin{align*}
    b_{w,\epsilon}(u):=\langle L_{w,n,\epsilon}u,u\rangle_{g^{p-r}},
\end{align*}
and 
\begin{align*}
    D_{w}(f):=
    \left\{
    \begin{aligned}
        &\int_{\mcal{X}}\|\nabla f_{g}(x)\|^{2}g(x)^{q}dx \qquad \text{if $f\in H^{1}(\mcal{X},g)$},\\
        &\infty\qquad\qquad\qquad\qquad\qquad \text{o.w.}
    \end{aligned}
    \right.
\end{align*}
Clearly, when $w=(p,q,r)=(1,2,0)$, the above Dirichlet energies become the ones associated with the unnormalized Laplacian, i.e., $w=(p,q,r)=(1,2,0)$:
\begin{align*}
    b_{\epsilon}(u):=\langle (\tilde{D}-\tilde{W})u,u\rangle,
\end{align*}
and 
\begin{align*}
    D_{2}(f):=
    \left\{
    \begin{aligned}
        &\int_{\mcal{X}}\|\nabla f(x)\|^{2}g(x)^{2}dx \qquad \text{if $f\in H^{1}(\mcal{X})$},\\
        &\infty\qquad\qquad\qquad\qquad\qquad \text{o.w.}
    \end{aligned}
    \right.
\end{align*}

The following two propositions from \cite{green2021minimaxsmoothing}, whose proof is based on Proposition \ref{transportdensity}, shows the fact that discretization map $\tilde{\mcal{P}}$ and interpolation map $\tilde{\mcal{I}}$ are almost isometries.

\begin{Proposition}[cf. Proposition 4 of \cite{green2021minimaxsmoothing}]\label{key1}
    With probability at least $1-C ne^{-C n \theta^{2} \tilde{\delta}^{d}}$, we have for any $f\in L^{2}(\mcal{X})$,
    \begin{align*}
        b_{\epsilon}(\tilde{\mcal{P}}f)\le C(1+C(\theta+\tilde{\delta}))\left(1+C \frac{\tilde{\delta}}{\epsilon}\right)\sigma_{1}\cdot D_{2}(f),
    \end{align*}
    and for any $u\in L^{2}(G_{n})$,
    \begin{align*}
        \sigma_1 D_{2}(\tilde{\mcal{I}}u)\le C(1+C(\theta+\tilde{\delta}))\left(1+C \frac{\tilde{\delta}}{\epsilon}\right)\cdot b_{\epsilon}(u).
    \end{align*}   
\end{Proposition}

\begin{Proposition}[cf. Proposition 5 of \cite{green2021minimaxsmoothing}]\label{key2}
    With probability at least $1-C ne^{-C n \theta^{2} \tilde{\delta}^{d}}$, we have for any $f\in L^{2}(\mcal{X})$,
    \begin{align*}
        \bigg|\|f\|_{L^{2}(G)}^2-\|\tilde{\mcal{P}}f\|_{L^{2}(G_{n})}^{2}\bigg|\le C\tilde{\delta}\|f\|_{L^{2}(G)}\sqrt{D_{2}(f)}+C(\theta+\tilde{\delta})\|f\|_{L^{2}(G)}^{2},
    \end{align*}
    and for any $u\in L^{2}(G_{n})$,
    \begin{align*}
        \bigg|\|u\|_{L^{2}(G_{n})}^2-\|\tilde{\mcal{I}}u\|_{L^{2}(G)}^{2}\bigg|\le C\epsilon\|u\|_{L^{2}(G_{n})}\sqrt{b_{\epsilon}(u)}+C(\theta+\tilde{\delta})\|u\|_{L^{2}(G_{n})}^{2}.
    \end{align*}  
\end{Proposition}

Now, as we consider the Dirichlet energies $b_{w,\epsilon}(u)$ and $D_{w}(f)$ for the weighted Laplacian. Note that by the boundedness assumption of the density $g$, we have there exist constants $C>0$ and $C'>0$ such that
\begin{align*}
    C'\int_{\mcal{X}}\|\nabla f_{g}(x)\|^{2}g(x)^{q}dx\le \int_{\mcal{X}}\|\nabla f_{g}(x)\|^{2}g(x)^{2}dx\le C\int_{\mcal{X}}\|\nabla f_{g}(x)\|^{2}g(x)^{q}dx.
\end{align*}
Also, with transformation $v:=D^{-r/(q-1)}u$ for $q\neq 1$, we have
\begin{align*}
    \langle L_{w,n,\epsilon}u,u\rangle_{g^{p-r}}=\langle (D-W)v,v\rangle. 
\end{align*}
This also holds for $q=1$ by definition \eqref{weightedgraph}. According to Section \ref{kdesectionpf}, we obtain that there exist constants $C>0$ and $C'>0$ such that for large $n$, almost surely,
\begin{align*}
    C'b_{w,\epsilon}(u)\le b_{\epsilon}(u)\le Cb_{w,\epsilon}(u).
\end{align*}

Consequently, following the proof in \cite{green2021minimaxsmoothing}, we present the following propositions parallelling Proposition \ref{key1} and \ref{key2} associated with the weighted case.

\begin{Proposition}\label{key3}
    With probability at least $1-C ne^{-C n \theta^{2} \tilde{\delta}^{d}}$, we have for any $f\in L^{2}(\mcal{X},g^{p-r}
    )$,
    \begin{align*}
        b_{\epsilon}(\tilde{\mcal{P}}f)\le C(1+C(\theta+\tilde{\delta}))\left(1+C \frac{\tilde{\delta}}{\epsilon}\right)\sigma_{1}\cdot D_{2}(f),
    \end{align*}
    and for any $u\in L^{2}(G_{n})$,
    \begin{align*}
        \sigma_1 D_{2}(\tilde{\mcal{I}}u)\le C(1+C(\theta+\tilde{\delta}))\left(1+C \frac{\tilde{\delta}}{\epsilon}\right)\cdot b_{\epsilon}(u),
    \end{align*}   
    where $C>0$ is some constant not depending on $n$ or $f$.
\end{Proposition}

\begin{Proposition}\label{key4}
    With probability at least $1-C ne^{-C n \theta^{2} \tilde{\delta}^{d}}$, we have for any $f\in L^{2}(\mcal{X},g^{p-r})$,
    \begin{align*}
        \bigg|\|f\|_{L^{2}(\mcal{X},g^{p-r})}^2-\|\tilde{\mcal{P}}f\|_{w,n}^{2}\bigg|\le C\tilde{\delta}\|f\|_{L^{2}(\mcal{X},g^{p-r})}\sqrt{D_{w}(f)}+C(\theta+\tilde{\delta})\|f\|_{L^{2}(\mcal{X},g^{p-r})}^{2}+\Delta(n,\epsilon,\eta,g)+\epsilon,
    \end{align*}
    and for any $u\in L^{2}(G_{n})$,
    \begin{align*}
        \bigg|\|u\|_{w,n}^2-\|\tilde{\mcal{I}}u\|_{L^{2}(\mcal{X},g^{p-r})}^{2}\bigg|\le C\epsilon\|u\|_{w,n}\sqrt{b_{w,\epsilon}(u)}+C(\theta+\tilde{\delta})\|u\|_{w,n}^{2}+\Delta(n,\epsilon,\eta,g)+\epsilon,
    \end{align*}
    where $C>0$ is some constant not depending on $n$ or $f$ and 
    \begin{align*}
    \Delta(n,\epsilon,\eta,g)&=\frac{1}{n}g_{\max}+\frac{\eta(0)}{n\epsilon^{d}}+\frac{n-1}{n}\Delta(n,\epsilon),\\
    \Delta(n,\epsilon)&:=\sqrt{\frac{|\log \epsilon|}{n\epsilon^{d}}}+\epsilon.
\end{align*}
\end{Proposition}

Also, we state the following Weyl's law whose proof follows \cite[Lemma 7.10]{dunlop2020large}.
\begin{Proposition}[Weyl's law]\label{refWeyl}
    There exist constants $C,C'>0$ such that
    \begin{align*}
        C'l^{2/d}\le \lambda_{l}(\mcal{L}_{w})\le Cl^{2/d},
    \end{align*}
    for all $l\ge 2$.
\end{Proposition}

Therefore, by following \cite[Proof of Lemma 2]{green2021minimaxsmoothing} except that we replace Propositions \ref{key1} and \ref{key2} by Propositions \ref{key3} and \ref{key4}, we obtain the following bound for the eigenvalues.

\begin{Lemma}\label{eigenref}
    Under the assumptions \ref{a1} and \ref{a2}, there exist constant $C,C'>0$ and $N>0$ such that for $n\ge N$ and $C(\log n/n)^{1/d}\le \epsilon\le C$, with probability larger than $1-Cne^{-Cn\epsilon^{d}}$, it holds that
    \begin{align*}
        C'\min\{l^{2/d},\epsilon^{-2}\}\le \lambda_{l}(L_{w,n,\epsilon})\le C\min\{l^{2/d},\epsilon^{-2}\},
    \end{align*}
    for all $2\le l\le n$.
\end{Lemma}

\end{document}